\documentclass[reqno,final]{amsart}
 
\usepackage{a4wide}

\usepackage[latin1]{inputenc}
\usepackage{amsmath,amssymb,bbm}

\usepackage{enumitem,color}

\definecolor{darkgreen}{rgb}{0.00,0.50,0.10}
\definecolor{lightgreen}{rgb}{0.20,0.70,0.30}

\usepackage[colorlinks,bookmarks,linkcolor=black,citecolor=black]{hyperref}

\newtheorem{theorem}                   {Theorem}
\newtheorem{lemma}           [theorem] {Lemma}   
   
\newtheorem{proposition}     [theorem] {Proposition}  
 
\newtheorem{fact}{Fact}
 
\newtheorem{definition}      [theorem] {Definition}

\newtheorem{question}        [theorem] {Question}
 \theoremstyle{remark}

\let\subset\subseteq
\let\eps\varepsilon
\let\rho\varrho

\def\dcup{\dot\cup}

\newcommand{\By}[2]{\overset{\mbox{\tiny{#1}}}{#2}}
\newcommand{\ByRef}[2]{   \By{\eqref{#1}}{#2} }

\newcommand{\leBy}[1]{    \By{#1}{\le} }

\newcommand{\gByRef}[1]{  \ByRef{#1}{>} }
\newcommand{\leByRef}[1]{ \ByRef{#1}{\le} }
\newcommand{\geByRef}[1]{ \ByRef{#1}{\ge} }

\newcommand{\subref}[2][L]{\text{\tiny #1\ref{#2}}}

\newcommand{\K}[2][n]{\mathcal{K}_{#1}^{#2}}
\newcommand{\neighbor}{\mathrm{N}}

\newcommand{\nR}{k}  

\newcommand{\Exp}{\mathbb{E}}
\newcommand{\Prob}{\mathbb{P}}

\newcommand{\comment}[1]{}

  \title{The tripartite Ramsey number for trees}

  \author{Julia B\"ottcher} 
  \address{Zentrum Mathematik, Technische Universit\"at M\"unchen,
    Boltzmannstra\ss{}e~3, D-85747 Garching bei M\"unchen, Germany} 
  \email{boettche@ma.tum.de}

  \author{Jan Hladk\'y}
  \address{Department of Applied Mathematics, Faculty of mathematics and Physics, Charles University, Malostransk\'e n\'am\v est\'i
 25, 118 00, Prague, Czech Republic \and Zentrum Mathematik,
 Technische Universit\"at M\"unchen, Boltzmannstra\ss{}e~3, D-85747 Garching bei M\"unchen, Germany}
  \email{honzahladky@googlemail.com}
    
  \author{Diana Piguet}
  \address{Zentrum Mathematik, Technische Universit\"at M\"unchen,
    Boltzmannstra\ss{}e~3, D-85747 Garching bei M\"unchen, Germany} 
  \email{diana@kam.mff.cuni.cz}

  \thanks{The first and the third author were partially supported by
  DFG grant TA 309/2-1. The second and the third author were partially
  supported by DAAD. The second author was partially supported by the
  grant GAUK 202-10/258009 of the Grant Agency of Charles University.}

\begin{document}

  \begin{abstract}
    We prove that for all $\eps>0$ there are $\alpha>0$ and $n_0\in \mathbb N$
    such that for all $n\ge n_0$ the following holds. For any two-colouring of the edges of
    $K_{n,n,n}$ one colour contains copies of all trees $T$ of order
    $t\le(3-\eps)n/2$ and with maximum degree $\Delta(T)\le n^\alpha$. This
    confirms a conjecture of Schelp.
  \end{abstract}

  \maketitle

\section{Introduction and results}
The celebrated theorem of Ramsey~\cite{Ramsey} states that for any finite family
of graphs~$\mathcal F$ the number $R(\mathcal F)$, defined as the smallest
integer $m$ such that in any edge-colouring of $K_m$ with green and red there
are either copies of all members of $\mathcal F$ in green or in red, exists. In
this case we also write $K_m\rightarrow \mathcal F$ and say that $K_m$ is
\emph{Ramsey} for $\mathcal F$. Let $\mathcal{T}_t$ denote the class of trees
of order $t$, $\mathcal{T}_t^\Delta$ is its restriction to trees of maximum
degree at most $\Delta$. 
Ajtai, Koml\'os, Simonovits, and Szemer\'edi~\cite{AKSS07+}
announced a result which implies that $K_{2t-2}\rightarrow\mathcal T_{t}$
for large even~$t$ and $K_{2t-3}\rightarrow\mathcal T_{t}$ for large
odd~$t$. This bound is best possible. For the case of odd~$t$ this is also
a consequence of a theorem by Zhao~\cite{Z07+}
concerning a conjecture of Loebl (see
also~\cite{HP08}).

The graph $K_{R(\mathcal F)}$ is obviously a Ramsey graph for $\mathcal F$ with
as few vertices as possible. However, one may still ask, whether there exist
graphs with fewer edges which are Ramsey for $\mathcal F$.
This minimal number of edges is also called {\it size Ramsey
number} and denoted by $R_s(\mathcal F)$. Trivially
$R_s(\mathcal F)\le \binom{R(\mathcal F)}{2}$, but it turns out that
this inequality is often far from tight. The investigation of size Ramsey
numbers recently experienced much attention. Trees are considered
in~\cite{Beck90,HK95}. Progress on determining the size Ramsey number for
classes of bounded degree graphs was made in~\cite{KRSSz08+}.

A question of similar flavour is what happens when we do not confine ourselves
to finding Ramsey graphs for $\mathcal F$ with few edges but require in addition
that they are proper subgraphs of $K_m$ with $m$ very close to $R(\mathcal
F)$. This question has two aspects: a quantitative one (i.e., how many edges can
be deleted from $K_m$ so that the remaining graph is still Ramsey) and
a structural one (i.e., what is the structure of the edges that may be deleted).
Questions of similar nature were explored in~\cite{GSS08+} when~$\mathcal F$ consists of an
odd cycle and in~\cite{GRSS07} when $\mathcal F$ is a path. Our focus in this
paper is on the case when $\mathcal F$ is a class of trees.

Schelp~\cite{SchelpConj} posed the following Ramsey-type conjecture
about trees in tripartite graphs:
For $n$ sufficiently large the tripartite graph $K_{n,n,n}$ is Ramsey for the
class~$\mathcal T_t^\Delta$ of trees on $t\le(3-\eps)n/2$ vertices with maximum
degree at most $\Delta$ for constant $\Delta$. The conjecture thus asserts that
we can delete three cliques of size $m/3$ from a graph $K_m$ with $m$ only
slightly larger than $R(\mathcal T_t^\Delta)$ while maintaining the Ramsey
property. In addition Schelp asked whether the same remains true when the
constant maximum degree bound in the conjecture above is replaced by
$\Delta\le\frac23 t$ (which is easily seen to be best possible). Our main
result is situated in-between these two cases, solving the problem for trees
of maximum degree $n^\alpha$ for some small~$\alpha$ and hence, in particular,
answering the first conjecture above.


\begin{theorem}\label{thm:main}
For all $\mu>0$ there are $\alpha>0$ and $n_0\in \mathbb N$
 such that for all $n\ge n_0$ 
 \[
 K_{n,n,n}\rightarrow \mathcal T_t^\Delta,
 \]
 with $\Delta\le n^\alpha$ and $t\le (3-\mu)n/2$.
\end{theorem}

We use Szemer\'edi's regularity
lemma~\cite{Sze78} to establish this result. Due to the nature of the methods
related to this lemma it follows that Theorem~\ref{thm:main} remains true when
$K_{n,n,n}$ is replaced by a much sparser graph: For any fixed $\mu\in(0,1]$ a
random subgraph of $K_{n,n,n}$ with edge probability~$\mu$ allows for the same
conclusion, as long as $n$ is sufficiently large (cf.\ Section~\ref{sec:remarks}).

The proof of Theorem~\ref{thm:main} splits into a combinatorial part
and a regularity based embedding part. The lemmas we need for the combinatorial
part are stated in Section~\ref{sec:match-fork} and proved in
Section~\ref{sec:comb}.   
\L{}uczak~\cite{Lucz99} first noted that a large
connected 
matching in a cluster graph is a suitable structure for embedding paths.
In the present paper, we
extend \L{}uczak's idea and use what we call
``odd connected matchings'' and ``connected
fork systems'' in the cluster graph.

For the embedding part we
formulate an embedding lemma (Lemma~\ref{lem:emb}, see Section~\ref{sec:proof}) that provides rather general conditions for the embedding of trees with growing maximum degree.
The proof of this lemma is prepared in Section~\ref{sec:valid} and presented in
Section~\ref{sec:emb}. First, however, we shall introduce all necessary
definitions as well as the regularity lemma in the following
section.


\section{Definitions and Tools}
\label{sec:def}

Let $G=(V,E)$ be a graph and $X,X',X''\subset V$ be pairwise disjoint vertex
sets. Then we define $E(X):=E\cap\binom{X}{2}$ and
$E(X,X'):=E\cap(X\times X')$ and write
$G[X]$ for the graph with vertex set~$X$ and edge set $E(X)$. Similarly,
$G[X,X']$ is the bipartite graph with vertex set $X\dcup X'$ and edge set
$E(X,X')$ and $G[X,X',X'']$ is the tripartite graph with vertex set $X\dcup
X'\dcup X''$ and edge set $E(X,X') \dcup E(X',X'') \dcup E(X'',X)$. For
convenience we frequently identify graphs~$G$ with their edge set~$E(G)$ and
vice versa. We say that a subgraph~$G'$ of~$G$ covers a vertex~$v$ of~$G$
if~$v$ is contained in some edge of $G'$. For a vertex set $D$ and an edge set
$M$ we denote by $D\cap M$ the set of vertices from $D$ that appear in some
edge of $M$. We write $\neighbor(v)$ for the neighborhood of a vertex $v$.

A \emph{matching}~$M$ in a graph~$G=(V,E)$ is a set of vertex disjoint edges
in~$E$ and its size is the number of edges in~$M$. 
For vertices~$v$ and vertex sets $U$ covered by~$M$ we also write, abusing
notation, $v\in M$ and
$U\subset M$. Sometimes we also consider a
matching as a bijection $M\colon V_M\to V_M$ where $V_M\subset V$ is the set of
vertices covered by $M$. For $U\subset V_M$ we then denote by $M(U)$ the set of
vertices $v\in V_M$ such that $uv\in M$ for some $u\in U$. 

To make our notation compact we sometimes use subscripts in a non-standard way
as illustrated by the following example. Let 
$A_1,A_2\subseteq A$ and $B_1,B_2\subseteq B$ be sets and suppose that $D\in
\{A,B\}$ and $i\in[2]$. The symbol $D_i$ then denotes the set $A_i$ if $D=A$
and the set $B_i$ if $D=B$.
  

\subsection{Regularity}


Let $G=(V,E)$ be a graph and
$\eps,d\in[0,1]$. For disjoint nonempty vertex sets $U,W\subset V$ the
\emph{density} $d(U,W)$ of the pair $(U,W)$ is the number of edges that run
between $U$ and $W$ divided by $|U||W|$. A pair $(U,W)$ with density at least
$d$ is \emph{$(\eps,d)$-regular} if $|d(U',W')-d(U,W)|\le\eps$ for all
$U'\subset U$ and $W'\subset W$ with $|U'|\ge\eps|U|$ and $|W'|\ge\eps|W|$. 
The following lemma states that in dense regular pairs most vertices have many
neighbours. This is an immediate consequence of the definition of regular pairs.

\begin{lemma}\label{lem:typical}
  Let $(U,U')$ be an $(\eps,d)$-regular pair and $X\subset U$ with
  $|X|\ge\eps|U|$. Then less than $\eps|U'|$ vertices in $U'$ have less than
  $(d-\eps)|X|$ neighbours in $X$.
\qed
\end{lemma}

In the rest of the paper we will say that all other vertices in $U$ are
\emph{$(\eps,d)$-typical} with respect to $X$ (or simply typical, when $\eps$
and $d$ are clear from the context).

An \emph{$(\eps,d)$-regular partition} of~$G=(V,E)$ with \emph{reduced graph}
$\mathbb G=([\nR],E_{\mathbb G})$ is a partition $V_0\dcup V_1 \dcup\dots\dcup
V_\nR$ of $V$ with $|V_0|\le\eps|V|$, such that $(V_i,V_j)$ is an $(\eps,d)$-regular
pair in $G$ whenever 
$ij\in E_{\mathbb G}$. In this case we also say that $G$ has
\emph{$(\eps,d)$-reduced graph}~${\mathbb G}$. ({Throughout this paper
blackboard symbols such as $\mathbb G$ or $\mathbb M$ denote reduced
graphs and their subgraphs.})
The partition
classes $V_i$ with $i\in[\nR]$ are also called \emph{clusters} of $G$ and $V_0$
is the \emph{bin set}. We also call a vertex $i$ of the reduced graph a cluster
and identify it with its corresponding set $V_i$.
 
Suppose that $P$ is a partition of $V$. We then say that a partition
$V_0\dcup V_1\dcup\ldots\dcup V_s$ of $V$ \emph{refines}~$P$ if for every
$i\in[s]$ there exists a member $A\in P$ such that $V_i\subset A$. Finally, a partition $V_0\dcup V_1 \dcup\dots\dcup
V_\nR$ of $V$ is an \emph{equipartition} if $|V_i|=|V_j|$ for all $i,j\in[\nR]$.

Now we state a version of Szemer\'edi's celebrated regularity
lemma~\cite{Sze78}. This lemma takes an $n$-vertex graph $G$ that is given with some preliminary
partition and produces a regular partition of $G$ with $\nR\le\nR_1$ clusters
which refines this partition where $\nR_1$ does not depend on $n$.

\begin{lemma}[Regularity lemma]
\label{lem:RL}
  For all $\eps>0$ and integers $\nR_0$ and $\nR_*$ there is an integer
  $\nR_1$ such that for all graphs $G=(V,E)$ on $n\ge\nR_1$ vertices the
  following holds. Let $G$ be given together with a partition
  $V=V^*_1\dcup\dots\dcup V^*_{\nR_*}$ of its vertices.
  Then there is $\nR_0\le\nR\le\nR_1$ such that $G$ has an $\eps$-regular
  equipartition $V=V_0\dcup V_1\dcup\dots\dcup V_{\nR}$ refining
  $V^*_1\dcup\dots\dcup V^*_{\nR_*}$
\end{lemma}

We also say that $V=V^*_1\dcup\dots\dcup V^*_{\nR_*}$ is a \emph{prepartition}
of~$G$.

\subsection{Coloured graphs}

A \emph{coloured graph}~$G$ is a graph $(V,E)$ together with a $2$-colouring
of its edges by red and green. We denote by~$G(c)$ the subgraph of~$G$ formed
by exactly those edges with colour~$c$. 
Two vertices are
\emph{connected} in $G$ if they lie in the same connected component of $G$ and
are \emph{$c$-connected} in $G$ if they are connected in $G(c)$.
Let~$G$ be a coloured graph and $v$ be a vertex of $G$ and
$c\in\{\text{red,\text{green}}\}$. Then a vertex $u$ is a
\emph{$c$-neighbour} of $v$ if $uv$ is an edge of colour $c$ in $G$. The
\emph{$c$-neighbourhood} of $v$ is the set of all $c$-neighbours of $v$.
    
\begin{definition}[connected, odd, even]
  Let $G'$ be either a subgraph of an uncoloured graph $G$, or a
  $c$-monochromatic subgraph of a coloured graph $G$. Then we say that $G'$ is
  \emph{connected} if any two vertices covered by $G'$ are connected,
  respectively $c$-connected, in $G$. Further, the component of $G$,
  respectively of $G(c)$, containing $G'$ is called the
  \emph{component of~$G'$} and is denoted by $G[G']$. Further,~$G'$ is
  \emph{odd} if there is an odd cycle in~$G[G']$, otherwise $G'$ is \emph{even}.
\end{definition}

  Notice that this notion of connected subgraphs differs from the standard one.
  A red-connected matching is a good example to illustrate
  this concept: it is a matching with all edges coloured in
  red and with a path (in the original graph) of red colour between any two
  vertices covered by the matching.
  For subgraphs containing edges of different colours the notion of
  connectedness is not defined.


\begin{definition}[fork, fork system]
  An \emph{$r$-fork} (or simply \emph{fork}) is the complete bipartite graph
  $K_{1,r}$. We also say that an $r$-fork has~$r$ \emph{prongs} and
  one~\emph{center} by which we refer to the vertices in the two partition
  classes of $K_{1,r}$. A \emph{fork system}~$F$ in a graph~$G$ is a set of
  pairwise vertex disjoint forks in~$G$ (not necessarily having the same number
  of prongs). We say that~$F$ has \emph{ratio}~$r$ if all its forks have at
  most $r$ prongs. Then we also call~$F$ an $r$-fork system.
\end{definition}

  Suppose~$F$ is a connected fork system in $G$. If $F$ is even then the
  \emph{size}~$f$ of $F$ is the order of the bigger bipartition class of $G[F]$.
  If~$F$ is odd then~$F$ has size at least~$f$ if there is a connected bipartite
  subgraph~$G'$ of~$G$ such that~$F$ has size~$f$ in~$G'$. 
  For a vertex set $D$ in $G$ we say that $F$ is \emph{centered} in $D$ if the
  centers of the forks in $F$ all lie in $D$.

  Next, we define two properties of coloured graphs that characterise
  structures (in a reduced graph) suitable for the embedding of trees as we
  shall see later (cf.\ Section~\ref{sec:proof}). Roughly speaking, 
  these properties guarantee the existence of large
  monochromatic connected matchings and fork-systems.

\begin{definition}[$m$-odd, $(m,f,r)$-good]
  Let $G$ be a coloured graph on $n$ vertices. Then~$G$ is called
  $m$-\emph{odd} if $G$ contains a monochromatic odd connected matching of
  size at least $m$. We say that~$G$ is
  \emph{$(m,f,r)$-good} (in colour $c$) if~$G$ contains a $c$-coloured
  connected matching $M$ of size at least~$m$ as well as a $c$-coloured
  connected fork system $F$ of size at least~$f$, and ratio at most~$r$.
\end{definition}

We further need to define a set of special, so-called extremal, configurations
of coloured graphs that will need special treatment in our proofs.
To prepare their definition, let $K$ be a graph on $n$ vertices and $D,D'$ be
disjoint vertex sets in $K$. We say that the bipartite graph $K[D,D']$ is
\emph{$\eta$-complete} if each vertex of $K[D,D']$ is incident to all but at
most $\eta n$ vertices of the other bipartition class.
If $K$ is a coloured graph then $K[D,D']$ is \emph{$(\eta,c)$-complete} for
some colour $c$ if it is $\eta$-complete and all edges in $K[D,D']$ are of
colour $c$. We call a set $A$ \emph{negligible} if $|A|<2\eta n$. Otherwise, $A$
is \emph{non-negligible}. 

\begin{definition}[extremal]
\label{def:ext}
  Let~$K=(V,E)$ be a coloured graph of order $3n$. Suppose that $\eta>0$ is
  given. We say that~$K$ is a pyramid configuration with parameter~$\eta$ if it
  satisfies~\ref{def:pyra} below and a spider configuration if it
  satisfies~\ref{def:spider}. In both cases we call $K$
 \emph{extremal with parameter $\eta$} or \emph{$\eta$-extremal}. Otherwise we
 say that~$K$ is \emph{not $\eta$-extremal}.
  \begin{enumerate}[label={\rm(E\arabic{*})},
  topsep=1mm,itemsep=0cm,parsep=0cm,
  leftmargin=*,labelsep=0.1cm]
    \item\label{def:pyra} \emph{pyramid configurations:}
      There are (not necessarily distinct) colours $c, c'$ and
      pairwise disjoint subsets $D_1, D_2, D'_1, D'_2\subseteq V$ of size at
      most $n$, with $|D_1|,|D_2|\ge(1-\eta)n$ and $|D'_1|+|D'_2|\ge(1-\eta)n$ where $D'_1$
      and $D'_2$ are either empty or non-negligible. Further,
      $K[D_1,D'_1]$ and $K[D_2,D'_2]$ are $(\eta,c)$-complete and
      $K[D_1,D'_2]$, $K[D_2,D'_1]$, and $K[D_1,D_2]$ are $\eta$-complete.
      
      In addition, either $K[D_1,D_2]$ is $(\eta,c')$-complete or both
      $K[D_1,D'_2]$ and $K[D'_1,D_2]$ are $(\eta,c')$-complete. In the first
      case we say the pyramid configuration has a~\emph{$c'$-tunnel}, and in the
      second case that it has a~\emph{crossing}. The pairs $(D_1,D'_1)$ and
      $(D_2,D'_2)$ are also called the~\emph{pyramids} of this configuration.
    \item\label{def:spider} \emph{spider configuration:}
      There is a colour $c$ and pairwise disjoint
      subsets $A_1, A_2$, $B_1, B_2, C_1, C_2\subseteq V$ such that
      $|D_1\cup D_2|\ge(1-\eta)n$ and $K[D_1,D'_2]$ is $(\eta,c)$-complete
      for all $D,D'\in\{A,B,C\}$ with $D\neq D'$, the edges in all these
      bipartite graphs together form a connected bipartite subgraph $K_c$ of $K$
      with (bi)partition classes $A_1\dcup B_1\dcup C_1$ and $A_2\dcup B_2\dcup
      C_2$. Further there are sets $A_B\dcup A_C=A_2$, $B_A\dcup B_C=B_2$, and
      $C_A\dcup C_B\dcup C_C=C_2$, each of which is either empty or
      non-negligible, such that the
      following conditions are satisfied for all $\{D,D',D''\}=\{A,B,C\}$:
      \begin{enumerate}[label={\rm \arabic{*}.},ref={\rm \arabic{*}}]
        \item\label{def:spider:1} $|A_1|\ge|B_1|\ge|C_1\cup C_C|$ and
          $|D_{D'}|=|D'_D|\le n-|D''_2|$,
        \item\label{def:spider:2} either $C_C=\emptyset$ or $A_B=\emptyset$,
        \item\label{def:spider:3} either $A_2=\emptyset$ or
        $|A_2\cup B_2\cup C_A\cup C_B|\le(1-\eta)\frac32n$,
        \item\label{def:spider:4} either $C_1=\emptyset$ or $|A_1\cup B_1\cup
          C_1|<(1-\eta)\frac32n$ or $|B_1\cup C_1|\le(1-\eta)\frac34n$.
      \end{enumerate}        
  \end{enumerate}
\end{definition}

By $\K{\eta}$, finally, we denote the class of all spanning subgraphs $K$ of
$K_{n,n,n}$ with minimum degree $\delta(K)>(2-\eta) n$. We also call the graphs in this
class \emph{$\eta$-complete tripartite graphs}.


\section{Connected matchings and fork systems}
\label{sec:match-fork}

In order to prove Theorem~\ref{thm:main} we will use the following structural
result about coloured graphs from~$\K{\eta}$. It asserts that such graphs
either contain large monochromatic odd connected matchings or appropriate
connected fork systems. With the help of the regularity method we will then, in
Section~\ref{sec:proof}, use this result (on the reduced graph of a regular
partition) to find monochromatic trees. The reason why odd connected matchings
and connected fork systems are useful for this task is explained in
Section~\ref{sec:proof:idea}.

\begin{lemma}
\label{lem:good-odd}
  For all~$\eta'>0$ there are~$\eta>0$ and $n_0\in \mathbb N$ such that for all
  $n\ge n_0$ the following holds. Every coloured graph $K\in\K{\eta}$ is either
  $(1-\eta')\frac34n$-odd or $\big((1-\eta')n,(1-\eta')\frac32n,3\big)$-good.
\end{lemma}

We remark that the dependence of the constant $n_0$ and $\eta'$ is
only linear, and in fact we can choose $n_0=\eta'/200$. 
As we will see below, Lemma~\ref{lem:good-odd} is a consequence of the following
two lemmas. The first lemma analyses non-extremal members of~$\K{\eta}$.

\begin{lemma}[non-extremal configurations]
\label{lem:odd}
  For all~$\eta'>0$ there are  $\eta\in (0,\eta')$ and $n_0\in \mathbb
  N$ such that for all $n\ge n_0$ the following holds. Let~$K$ be a coloured
  graph from~$\K{\eta}$ that is not $\eta'$-extremal. Then~$K$ is
  $(1-\eta')\frac34n$-odd.
\end{lemma}

The second lemma handles the extremal configurations.

\begin{lemma}[extremal configurations]
\label{lem:ext}
  For all~$\eta'>0$ there is $\eta\in (0,\eta')$ such that the following holds.
  Let~$K$ be a coloured graph from~$\K{\eta}$ that is $\eta$-extremal. Then~$K$
  is $((1-\eta')n,(1-\eta')\frac32n,3)$-good.
\end{lemma}

Proofs of Lemma~\ref{lem:odd} and~\ref{lem:ext} are provided in
Sections~\ref{sec:Nextr-conf} and~\ref{sec:ext}, respectively. We get
Lemma~\ref{lem:good-odd} as an easy corollary.

\begin{proof}[Proof of Lemma~\ref{lem:good-odd}]
  Given~$\eta'$ let $\eta_{\subref{lem:ext}}<\eta'$ be the constant provided by
  Lemma~\ref{lem:ext} for input $\eta'$ and let $\eta_{\subref{lem:odd}}$ be the
  constant produced by Lemma~\ref{lem:odd} for input
  $\eta'_{\subref{lem:odd}}:=\eta_{\subref{lem:ext}}$. Set
  $\eta:=\min\{\eta_{\subref{lem:ext}},\eta_{\subref{lem:odd}}\}$ and let
  $K\in\K{\eta}$ be a given coloured graph. Then
  $K\in\K{\eta_{\subref{lem:odd}}}$ and by Lemma~\ref{lem:odd} the graph $K$ is
  either $(1-\eta'_{\subref{lem:odd}})3n/4$-odd (and thus $(1-\eta')3n/4$-odd
  as oddness is monotone) or $\eta_{\subref{lem:ext}}$-extremal. In the first
  case we are done and in the second case Lemma~\ref{lem:ext} implies that $K$ is
  $\left((1-\eta')n,(1-\eta')\frac32 n,3\right)$-good (goodness is also
  monotone) and we are also done.
\end{proof}


\section{Proof of theorem~\ref{thm:main}}
\label{sec:proof}

In this section we will first briefly outline the main ideas for the proof of
Theorem~\ref{thm:main}. Then we will state the remaining necessary lemmas, most
notably our main embedding result (Lemma~\ref{lem:emb}). These lemmas will be
proved in the subsequent sections. At the end of this section we finally
provide a proof of Theorem~\ref{thm:main}.

\subsection{The idea of the proof}
\label{sec:proof:idea}
We apply the Regularity Lemma on the coloured graph $K_{n,n,n}$ with
prepartition as given by the partition classes of $K_{n,n,n}$. As a result we
obtain a coloured reduced graph $\mathbb{K}\in \mathcal{K}^{\eta}_k$ where
the colour of an edge in~$\mathbb{K}$ corresponds to the majority colour
in the underlying regular pair. Such a regular pair is well-known to possess almost as good
embedding properties as a complete bipartite graph. We apply our structural
result Lemma~\ref{lem:good-odd} and infer that~$\mathbb{K}$ is either
  $(1-\eta')\frac34k$-odd or $((1-\eta')k,(1-\eta')\frac32k,3)$-good, i.e.,
  there is a colour, say green, such that $\mathbb{K}$ contains either an odd connected
  green matching $\mathbb{M}_o$ of size at least $(1-\eta')\frac34k$, or it
  contains a connected greed matching $\mathbb{M}$ of size at least
  $(1-\eta')k$ and a $3$-fork system $\mathbb{F}$ of size at least
  $(1-\eta')\frac32k$. We shall show that using either of these structures we
  can embed any tree $T\in\mathcal{T}^\Delta_t$ into the green subgraph of
  $K_{n,n,n}$. As a preparatory step, we cut $T$ into small subtrees (see
  Lemma~\ref{lem:cut}), called shrubs.
  
  Now let us first consider the case when we have an odd
  matching~$\mathbb{M}_o$. Our aim is to embed each shrub~$S$ into a regular
  pair~$(A,B)$ corresponding to an edge $e\in\mathbb{M}_o$. Shrubs are
  bipartite graphs. Therefore there are two ways of assigning the colour
  classes of~$S$ to the clusters of~$e$. This corresponds to two different
  \emph{orientations} of~$S$ for the embedding in $(A,B)$. Our strategy is
  to choose orientations for all shrubs (and hence assignments of their
  colour classes to clusters of edges in~$\mathbb M_o$) in such a way that every
  cluster of $V(\mathbb M_o)$ receives roughly the same number of vertices of
  $T$. We will show that this is possible without ``over-filling'' any
  cluster. It follows that we can embed all shrubs into regular pairs
  corresponding to edges of~$\mathbb M_o$. The fact that $\mathbb{M}_o$ is
  connected and \emph{odd} then implies that between any pair of edges
  in~$\mathbb M_o$ there are walks of both even and odd length in the
  reduced graph. We will show that this allows us to connect the shrubs and to
  obtain a copy of~$T$ in the green subgraph of $K_{n,n,n}$.
  
  For the second case, i.e., the case when we have a
  matching~$\mathbb M$ as well as a $3$-fork system~$\mathbb F$ the basic
  strategy remains the same. We assign shrubs to edges of~$\mathbb M$
  or~$\mathbb F$. In difference to the previous case, however, these
  substructures of the reduced graph are not odd. This means that we cannot
  choose the orientations of the shrubs as before. Rather, these orientations
  are determined by the connections between the shrubs. Therefore, we
  distinguish the following two situations when embedding the tree~$T$. If
  the partition classes of~$T$ are reasonably balanced, then we use the
  matching~$\mathbb M$ for the embedding. If~$T$ is unbalanced, on the other
  hand, we employ the fork system~$\mathbb F$ and use the prongs of the
  forks in $\mathbb F$  to accommodate the bigger partition class of $T$ and
  the centers for the smaller.


\subsection{The main embedding lemma}
\label{sec:proof:emb}

As indicated, in the proof of the main theorem we will use the regularity lemma
in conjunction with an embedding lemma (Lemma~\ref{lem:emb}). This lemma states
that a tree~$T$ can be embedded into a graph given together with a regular
partition if there is a homomorphism from~$T$ to the reduced graph of the
partition with suitable properties. In the following definition of a valid
assignment we specify these properties. 
Roughly speaking, a valid assignment is a homomorphism $h$
from a tree $T$ to a (reduced) graph $\mathbb G$ such that no vertex
of $\mathbb G$ receives too many vertices of $T$ and that does not
``spread'' in the tree too quickly in the following sense: for each vertex
$x\in V(T)$ we require that the neighbours of~$x$ occupy at most two vertices of
$\mathbb G$.

\begin{definition}[valid assignment]\label{def:valid}
Let $\mathbb G$ be a graph on vertex set $[\nR]$, let $T$ be a tree,
$\rho\in[0,1]$ and $L\in\mathbb N$. A mapping $h\colon V(T)\rightarrow [\nR]$ is
a \emph{$(\rho,L)$-valid assignment} of $T$ to $\mathbb G$ if 
\begin{enumerate}[label={\rm \arabic{*}.},ref={\rm \arabic{*}}]  
  \item \label{def:valid:hom} $h$ is a homomorphism from $T$ to
  $\mathbb G$,
  \item \label{def:valid:N}$|h(\neighbor_T(x))|\le 2$, for every $x\in V(T)$,
  \item \label{def:valid:place} $|h^{-1}(i)|<(1-\rho)L$, for every $i\in[\nR]$.
\end{enumerate}
\end{definition}

In addition we need the concept of a cut of a tree, which is a set of
vertices that cuts the tree into small components which we call shrubs.

\begin{definition}[cut, shrubs]
\label{def:S-cut}
  Let $S\in \mathbb N$ and~$T$ be a tree with vertex set $V(T)$. A set
  $C\subseteq V(T)$ is an \emph{$S$-cut} (or simply \emph{cut}) of~$T$ if all
  components of $T-C$ are of size at most~$S$. The components of $T-C$ are
  called the \emph{shrubs} of~$T$ corresponding to~$C$.
\end{definition}

Now we can state the main embedding lemma.

\begin{lemma}[main embedding lemma]
\label{lem:emb}
  Let $G$ be an $n$-vertex graph with an $(\eps,d)$-reduced graph $\mathbb
  G=([\nR],E(\mathbb G))$ and let $T$ be a tree with $\Delta(T)\le\Delta$ and
  an $S$-cut $C$. If $T$ has a $\big(\rho, (1-\eps)\frac n\nR\big)$-valid
  assignment to $\mathbb G$ and
  $(\frac1{10}d\rho-10\eps)\frac{n}{\nR}\ge|C|+S+\Delta$ then $T\subset G$.
\end{lemma}

The proof of this lemma is deferred to Section~\ref{sec:emb}.
Before we can apply it for embedding a tree~$T$ in the proof of
Theorem~\ref{thm:main} we need to construct a valid assignment for~$T$.
This is taken care of by the following lemma which states that
this is possible
if the reduced graph of some regular partition contains an odd connected
matching or a suitable fork system.
The proof of this lemma is given in Section~\ref{sec:valid}. 

\begin{lemma}[assignment lemma]
\label{lem:valid}
For all $\eps,\mu>0$ with $\eps<\mu/10$ and for all $\nR\in \mathbb N$ there is
 $\alpha=\alpha(\nR)>0$ and $n_0=n_0(\mu,\eps,\nR)\in \mathbb N$ such that for 
all $n\ge n_0$, all $r\in \mathbb N$, all graphs $\mathbb G$ of order~$\nR$,
and all trees $T$ with $\Delta(T)\le n^\alpha$ the following holds. Assume that either
\begin{enumerate}[label={\rm(\Alph{*})},start=13,leftmargin=*]
  \item\label{lem:valid:matching}
   $\mathbb G$ contains an odd connected matching of size at least $m$ and 
   that $t:=|V(T)|\le (1-\mu)2m\frac n\nR$, or
  \setcounter{enumi}{5}
  \item\label{lem:valid:fork}
  $\mathbb G$ contains a connected fork system with ratio $r$ and size
  at least $f$, and $T$ has colour class sizes $t_1$ and $t_2$ with
  $t_2\le t_1\le t'$ and $t_2\le t'/r$, where $t'=(1-\mu)f\frac n\nR$.
\end{enumerate}
Then there is an $(\eps\frac n\nR)$-cut $C$ of $T$ with $|C|\le \eps\frac n\nR$
and a $\big(\frac12\mu,(1-\eps)\frac n\nR\big)$-valid assignment of~$T$ to $\mathbb
G$.
\end{lemma}


\subsection{The proof}
\label{sec:proof:proof}

Now we have all tools  we need to prove the main theorem.

\begin{proof}[Proof of Theorem~\ref{thm:main}]
  We start by defining the necessary constants.
  Given $\mu>0$, set $\mu':=\eta'$ in such a way that
  \begin{equation}\label{eq:main:mu}
    1-\tfrac{\mu}3\le(1-\eta')^2(1-\mu').
  \end{equation}
  Lemma~\ref{lem:good-odd} with input $\eta'>0$ provides
  us with $\eta>0$ and $\nR_0\in \mathbb N$.  The regularity lemma,
  Lemma~\ref{lem:RL}, with input
  \begin{equation}\label{eq:main:eps}
    \eps:=\min\{\tfrac1{100}\eta^2, \tfrac1{10}\eta'^2, 10^{-3}\mu'\}    
  \end{equation}
  and $\nR_0$ and $\nR_*:=3$ returns a constant $\nR_1$.
  Next we apply Lemma~\ref{lem:valid} with input $\frac \eps {10}$ and $\mu'$ 
  separately for each value $3\nR$ with $\nR_0\le 3\nR\le\nR_1$ and get
  constants $\alpha(3\nR)$ and $n'_0(3\nR)$ for each of these applications. 
  Set $\alpha:=\min\{\alpha(3\nR)\colon\,\nR_0\le 3\nR\le\nR_1\}$ and
  $n'_0:=\max\{n'_0(3\nR)\colon\,\nR_0\le 3\nR\le \nR_1\}$.
  Finally, choose
  \begin{equation}\label{eq:main:n}
    n_0:=\max\{n'_0,\nR_1,(\tfrac {\nR_1}{\eps})^{1/(1-\alpha)}\}\,. 
  \end{equation}

  We are given a complete tripartite graph $K_{n,n,n}$ with $n\ge n_0$ as
  input whose edges are coloured with green and red. Our goal is to select a
  colour and show that in this colour we can embed every member of $\mathcal
  T_t^\Delta$ with $\Delta\le n^\alpha$ and $t\le (3-\mu)n/2$. 
  
  We first  select the colour. To this end let $G$ and $R$ be the subgraphs of
  $K_{n,n,n}$ formed by the green and red edges, respectively. We apply the
  regularity lemma, Lemma~\ref{lem:RL}, with input $\frac \eps {10}$ on the graph
  $G$ with prepartition $V^*_1\dcup V^*_2\dcup V^*_3$ as given by the three
  partition classes of $K_{n,n,n}$. We obtain an $\frac \eps {10}$-regular
  equipartition $V=V_0\dcup V_1\dcup\dots\dcup V_{3\nR}$ refining this
  prepartition such that $\nR_0\le 3\nR\le \nR_1$. Each cluster of this partition
  lies entirely in one of the partition classes of $K_{n,n,n}$. Let $\mathbb
  K=([3\nR],E_{\mathbb K})$ be the graph that contains edges for all
  $\eps$-regular cluster pairs that do not lie in the same partition class of
  $K_{n,n,n}$. Clearly, $\mathbb K$ is a tripartite graph. Furthermore, there are
  less than $\eps \nR^2$ pairs $(V_i,V_j)$ in our regular partition that are not
  $\frac \eps{10}$-regular in $G$. It follows that at most $2\sqrt{\eps} \nR$
  clusters $V_i$ are contained in more than $\sqrt{\eps} \nR$ irregular pairs. We
  move all these clusters and possibly up to $6\sqrt{\eps} \nR$ additional
  clusters to the bin set $V_0$. The additional clusters are chosen in such a way
  that we obtain in each partition class of $\mathbb K$ the same number of
  clusters. We call the resulting bin set $V'_0$ and denote the remaining
  clusters by $V'_1\dcup\dots\dcup V'_{3\nR'}$ and the corresponding subgraph of
  $\mathbb K$ by $\mathbb K'$. Observe that $\nR'\ge(1-3\sqrt{\eps})\nR$. Because
  each remaining cluster forms an irregular pair with at most $\sqrt{\eps}\nR\le
  2\sqrt{\eps}\nR'\le\eta'\nR'$ of the remaining clusters we conclude that
  $\mathbb K'$ is a graph from $\K[\nR']{\eta}$. In addition, it easily follows
  from the definition of $\eps$-regularity that each pair $(V_i,V_j)$ with $i$,
  $j\in[\nR']$ which is $\eps$-regular in $G$ is also $\eps$-regular in $R$. This
  motivates the following ``majority'' colouring of $\mathbb K'$: We colour the
  edges $ij$ of $\mathbb K'$ by green if the $\eps$-regular pair $(V_i,V_j)$ has
  density at least $\frac12$ and by red otherwise. In this way we obtain a
  coloured graph $\mathbb K'_c \in\K[\nR']{\eta}$.
  
  Now we are in a position to apply Lemma~\ref{lem:good-odd} to $\mathbb K'_c$.
  This lemma asserts that $\mathbb K'_c$ is either
  $(1-\eta')\frac34\nR'$-odd or
  $\big((1-\eta')\nR',(1-\eta')\frac32\nR',3\big)$-good. By definition this
  means that in one of the colours of $\mathbb K'_c$, say in green, we 
  \begin{enumerate}[label={\rm(\Alph{*})},start=15,leftmargin=*]
    \item\label{lem:main:o}
     either have an odd connected matching $\mathbb M_o$ of size
     $
      m_1\geq (1-\eta')\tfrac{3}{4}\nR'\ge
      ({1-\eta'})({1-3\sqrt{\eps}})\tfrac34\nR,
     $
    \setcounter{enumi}{6}
    \item\label{lem:main:g} 
    or a connected matching $\mathbb M$ of size
    $
      m_2\ge (1-\eta')\nR'\ge (1-\eta')(1-3\sqrt{\eps})\nR
    $
    together with a connected fork system $\mathbb F$ of size 
    $
      f\ge(1-\eta')\tfrac32\nR'\ge(1-\eta')(1-3\sqrt{\eps})\tfrac32\nR
    $
    and ratio $3$.
  \end{enumerate}
  In the following we  use  the matchings and
  fork systems we just obtained to show that we can embed all trees of
  $\mathcal T_t^\Delta$ in the corresponding system of regular pairs. For this 
  purpose let $\mathbb G$ be the graph on vertex set $[3k]$ that contains
  precisely all green edges of $\mathbb K'_c$.
   Observe that
   $\mathbb G$ is an $(\eps,1/2)$-reduced graph for $G$.
    
  Let $T\in\mathcal T_t^\Delta$ be a tree of order
  $t\le(3-\mu)n/2$
  and with maximal
  degree $\Delta(T)\le n^{\alpha}$. Now we distinguish two cases, depending on
  whether we obtained configuration~\ref{lem:main:o} or
  configuration~\ref{lem:main:g} from Lemma~\ref{lem:good-odd}.
  In both cases we plan to appeal to Lemma~\ref{lem:valid} to show that 
  $T$ has
  \begin{equation}\label{eq:main:valid}
    \text{an $(\eps\tfrac{n}{\nR})$-cut $C$ with $|C|\le \eps\tfrac{n}{\nR}$
    and a $(\tfrac12\mu',(1-\eps)\tfrac {3n}{3\nR})$-valid assignment 
    to $\mathbb G$.}
  \end{equation}
  Recall that we fed constants $\eps$, $\mu'>0$ and
  $3\nR$ into this lemma.
  Assume first that we are in configuration~\ref{lem:main:o}. 
  Because $m_1\ge({1-\eta'})({1-2\sqrt{\eps}})\tfrac34\nR$ 
  we have
  $$
  t\le(3-\mu)\frac{n}{2}
  \le 3(1-\tfrac{\mu}{3})\frac{n}{2}
  \frac{m_1}{(1-\eta')(1-3\sqrt{\eps})\tfrac34\nR}
  \leBy{\eqref{eq:main:mu},\eqref{eq:main:eps}}
  (1-\mu')2m_1 \cdot\frac {3n}{3\nR}\,.
  $$
  Hence by~\ref{lem:valid:matching} of Lemma~\ref{lem:valid} applied with the
  matching $\mathbb M_o$ (with~$n$ replaced by $\tilde n:=3n$ and~$k$ replaced
  by $\tilde k:=3k$) we get~\eqref{eq:main:valid} for $T$ in this case, as
  $\Delta(T)\le n^\alpha\le \tilde n^\alpha$.
  
  If we are in configuration~\ref{lem:main:g}, on the other hand,  then
  let $t_1\ge t_2$ be the sizes of the two colour classes of $T$. We
  distinguish two cases, using the two different structures provided
  in~\ref{lem:main:g}. Assume first that $t_2\le\frac
  t3$. Then, we calculate similarly as above that
  \begin{equation*}
    t_2 \le\tfrac13t
    \le(1-\tfrac13\mu)\tfrac n2
    \le\tfrac 13(1-\mu') f\tfrac{3n}{3\nR}\,,
    \qquad\text{and}\qquad
    t_1 \le t\le (1-\mu')f \tfrac {3n}{3\nR}\,.
  \end{equation*}
  Otherwise, if $t_2\ge\frac t3$ then, similarly,
  $$
  t_2\le t_1\le\tfrac23t\le(1-\tfrac \mu 3)n
  \le 
  (1-\mu') m_2\cdot \tfrac{3n}{3\nR}\,.
  $$
  Consequently, in both cases
  we can appeal to~\ref{lem:valid:fork} of Lemma~\ref{lem:valid}, in the first
  case applied to $\mathbb F$ and in the second to $\mathbb M$. We
  obtain~\eqref{eq:main:valid} for~$T$ as desired.

  We finish our proof with an application of the main embedding lemma,
  Lemma~\ref{lem:emb}. 
  As remarked earlier $\mathbb G$ is an
  $(\eps,1/2)$-reduced graph for $G$.
  We further have~\eqref{eq:main:valid}. For applying Lemma~\ref{lem:emb} it
  thus remains to check that
  $(\frac12\cdot\frac1{10}\rho-10\eps)\tfrac{n}{\nR}\ge S+|C|+\Delta$ with
  $\rho=\tfrac12\mu'$, $S=\eps \tfrac{n}{\nR}$, $|C|\le \eps \tfrac{n}{\nR}$,
  and $\Delta\le n^\alpha$. Indeed, 
  $$
    \left(\tfrac1{20}\rho-10\eps\right)\tfrac{n}{\nR}=
    \left(\tfrac 1{20}\cdot \tfrac12\mu'-10\eps\right)\tfrac{n}{\nR}
    \geByRef{eq:main:eps} 3\eps \tfrac{n}{\nR}
    \geByRef{eq:main:n} \eps \tfrac{n}{\nR}+\eps \tfrac{n}{\nR}+ n^\alpha\,.
  $$  
  So Lemma~\ref{lem:emb} ensures that $T\subset G$, i.\,e., there is an
  embedding of $T$ in the subgraph induced by the green edges in $K_{n,n,n}$.
\end{proof}


\section{Valid Assignments}
\label{sec:valid}

In this section we will provide a proof for Lemma~\ref{lem:valid}. The
idea is as follows. 
Given a tree~$T$ and a graph~$G$ with reduced graph~$\mathbb G$ we first
construct a cut of~$T$ that provides us with a collection of small shrubs
(see Lemma~\ref{lem:cut}). Then
we distribute these shrubs to edges of the given matching or fork-system
in~$\mathbb G$ (see Lemmas~\ref{lem:assign:matching}
and~\ref{lem:assign:fork}). Finally, we slightly modify this assignment in order
to obtain a homomorphism from~$T$ to $\mathbb G$ that satisfies the conditions
required for a valid assignment (see Lemma~\ref{lem:change-assignment}).

\begin{lemma}\label{lem:cut}
  For every $S\in \mathbb N$ and for any tree $T$ there is an $S$-cut of $T$
  that has size at most $\frac {|V(T)|}{S}$.
\end{lemma}
\begin{proof}
\setcounter{fact}{0}
To prove Lemma~\ref{lem:cut} we need the following fact.
\begin{fact}\label{fac:cut-vertex}
For any $S\in \mathbb N$ and any tree $T$ with $|V(T)|>S$, there is a vertex
$x\in V(T)$ such that the following holds.
If $F_x$ is the forest consisting of all components of $T-x$ with size at most
$S$, then $|V(F_x)|+1>S$.
\end{fact}
To see this, root the tree $T$ at an arbitrary vertex $x_0$. If $x_0$ does not
have the required property, it follows from $|V(T)|>S$ that
there exists a component $T_1$ in $T-x_0$ with $|V(T_1)|>S$. Set $x_1:=\neighbor(x_0)\cap V(T_1)$. Let
$F(T_1-{x_1})$ be the forest consisting of the components of $T_1-x_1$
that have size at most $S$. Observe that $F(T_1-x_1)$ is a subgraph 
of $F_{x_1}$. So if $|F(T_1-{x_1})|+1>S$, then $x_1$ has the property required by
Fact~\ref{fac:cut-vertex}. Otherwise there exists a component $T_2$ in
$T_1-x_1$ of size larger than $S$. Observe that $T_2$ is also a component of
$T-x_1$. 
Now repeat the procedure just described by setting 
$x_{2}=\neighbor(x_1)\cap V(T_{2})$ and so on, i.e., more generally we obtain
trees~$T_i$ and vertices $x_{i}=\neighbor(x_{i-1})\cap V(T_{i})$
. As the size of $T_i$ decreases as $i$ increases, there must
be an $x_i$ with the property required by Fact~\ref{fac:cut-vertex}.

Now we prove Lemma~\ref{lem:cut}. Set $C=\emptyset$. Repeat the following
process until it stops. Choose a component $T'$ of $T-C$ with size larger
than $S$. Apply Fact~\ref{fac:cut-vertex} to $T'$ and obtain a cut vertex $x$
of~$T'$ together with a forest $F_x$ consisting of components of $T'-x$ that
have size at most $S$ and is such that $|V(F_x)\cup \{x\}|>S$. Add $x$ to $C$
and repeat unless there is no component of size larger than $S$ in $T-C$. As
$|V(T-C)|$ decreases this process stops. Observe that
then $C$ is an $S$-cut. By the choice of $C$ we obtain
\[|V(T)|=\sum_{x\in C}|V(F_x)\cup \{x\}|>|C|\cdot S,\] which implies the
required bound on the size of $C$. 
\end{proof}

After Lemma~\ref{lem:cut} provided us with a cut and some corresponding shrubs
we will distribute each of these shrubs $T_i$ to an edge~$e$ of the odd matching
or the fork system in the reduced graph by assigning  one
colour class of $T_i$ to one end of~$e$ and the other colour class to
the other end. Here our goal is to distribute the shrubs and their vertices in such
a way that no cluster receives too many vertices. 
The next two lemmas guarantee that this can be done.
Lemma~\ref{lem:assign:matching} takes care of the distribution of the shrubs
to the clusters of a matching $M$ and
Lemma~\ref{lem:assign:fork} to those of a fork system $F$.
We provide Lemmas~\ref{lem:assign:matching} and~\ref{lem:assign:fork} with
numbers $a_{i,1}$ and $a_{i,2}$ as input. These numbers
represent the sizes of the colour classes $A_{i,1}$ and $A_{i,2}$ of the
shrub~$T_i$. Since we do not need any other information about the shrubs in
these lemmas the shrubs~$T_i$ do not explicitly appear in their statement.
Both lemmas then produces a mapping $\phi$ representing the assignment of the
colour classes~$A_{i,1}$ and~$A_{i,2}$ to the clusters of~$M$ or~$F$.

\begin{lemma}\label{lem:assign:matching}
  Let $\{a_{i,j}\}_{i\in [s],\; j\in [2]}$ be natural numbers with sum at most
  $t$ and $a_{i,1}+a_{i,2}\le S$ for all $i\in [s]$, and let $M$ be a
  matching on vertices $V(M)$. Then there is a mapping
  $\phi\colon [s]\times[2]\rightarrow V(M)$ such that $\phi(i,1)\phi(i,2)\in M$ for all $i\in[s]$ and
   \begin{equation}\label{eq:assign:matching}
     \sum_{(i,j)\in \phi^{-1}(v)}\!\!\!\!\!\!\!\!\!
       a_{i,j} \le\frac{t}{2|M|}+2S
     \qquad\qquad\text{for all $v\in V(M)$.}
   \end{equation}
  \end{lemma}
\begin{proof}
  A simple greedy construction gives the mapping $\phi$:
  We consider the numbers $a_{i,j}$ as weights that are distributed, first among
  the edges, and then among the vertices of $M$. For this purpose greedily
  assign pairs $(a_{i,1},a_{i,2})$ to the edges of $M$, in each step choosing an edge
  with minimum total weight. Then, clearly, no edge receives weight more than
  $S+t/|M|$. In a second round, do the following for each edge $vw$ of $M$.
  For the pairs $(a_{i,1},a_{i,2})$ that were assigned to $e$, greedily assign
  one of the weights of this pair to $v$ and the other one to $w$, such that
  the total weight on $v$ and on $w$ are as equal as possible. Hence each of
  these vertices receives weight at most $\frac12(S+t/|M|)+S$ and so the
  mapping $\phi$ corresponding to this weight distribution satisfies the
  desired properties.
\end{proof}

\begin{lemma}\label{lem:assign:fork}
  Let $\{a_{i,1}\}_{i\in [s]}$ and $\{a_{i,2}\}_{i\in [s]}$ be natural numbers
  with sum at most $t_1$ and $t_2$, respectively. Let $S\le t_1+t_2=:t$ and
  assume that $a_{i,1}+a_{i,2}\le S$ for all $i\in [s]$. Let $F$ be a fork system with ratio at most $r$ and partition classes $V_1(F)$ and $V_2(F)$ where
  $|V_1(F)|\ge|V_2(F)|$. Then there is a mapping $\phi\colon [s]\times[2]\rightarrow
  V_1(F)\cup V_2(F)$ such that $\phi(i,1)\phi(i,2)\in F$ and $\phi(i,j)\in
  V_j(F)$ for all $i\in[s], j\in[2]$ satisfying that for all $v_1\in V_1(F)$,
  $v_2\in V_2(F)$ we have
  \begin{equation}\label{eq:assign:fork}
   \sum_{(i,1)\in\phi^{-1}(v_1)}\!\!\!\!\!\!\!\!\!
      a_{i,1}\le\frac{t_1}{|F|}+\sqrt{12tS|F|}
   \qquad\text{and}\qquad
    \sum_{(i,2)\in\phi^{-1}(v_2)}\!\!\!\!\!\!\!\!\!
      a_{i,2}\le\frac{r t_2}{|F|}+\sqrt{12tS|F|}\,.
  \end{equation}
\end{lemma}

In the proof of this lemma we will make use the so-called Hoeffding bound for
sums of independent random variables (see,
e.g., \cite[Theorem~A.1.16]{AS00}).

\begin{theorem}
\label{thm:hoeff}
  Let $X_1,\dots,X_s$ be  independent random variables with $\Exp X_i=0$ and
  $|X_i|\le 1$ for all $i\in[s]$ and let $X$ be their sum. Then
  $\Prob[X>a]\le\exp(-a^2/(2s))$.
\qed
\end{theorem}

\begin{proof}[Proof of Lemma~\ref{lem:assign:fork}]
  For showing this lemma we use a probabilistic argument and again consider the
  $a_{i,j}$ as weights which are distributed among the vertices of $F$. 
  
  Observe first that we can assume without loss of generality that for all but
  at most one $i\in[s]$ we have $\frac12S\le a_{i,1}+a_{i,2}$ since otherwise
  we can group weights $a_{i,1}$ together, and also group the corresponding
  $a_{i,2}$ together, such that this condition is satisfied and continue with
  these grouped weights. This in turn implies, that $s\le(2t/S)+1\le 3t/S$. 
  
  We start by assigning weights $a_{i,1}$ to vertices of
  $V_1(F)$ by (randomly) constructing a mapping
  $\phi_1\colon [s]\times[1]\to V_1(F)$. To this end, independently and uniformly at random choose for each $i\in[s]$ an image
  $\phi_1(i,1)$ in $V_1(F)$. Clearly, there is a unique way of extending such
  a mapping $\phi_1$ to a mapping $\phi\colon [s]\times[2]\rightarrow V_1(F)\cup
  V_2(F)$ satisfying $\phi(i,1)\phi(i,2)\in F$. We claim that the probability
  that $\phi_1$ gives rise to a mapping $\phi$ which satisfies the assertions
  of the lemma is positive.
  
  Indeed, for any fixed vertex $v=v_1\in V_1(F)$ or $v=v_2\in V_2(F)$ let
  $\sigma(v)$ be the event that the mapping $\phi$ does not
  satisfy~\eqref{eq:assign:fork} for $v$.
  We will show that $\sigma(v)$ occurs with probability strictly less than
  $1/(2|F|)$, which clearly implies the claim above. We first
  consider the case $v=v_1\in V_1(F)$. For each $i\in[s]$ let $\mathbbm{1}_i$
  be the indicator variable for the event $\phi(i,1)=v_1$ and define a random
  variable $X_i$ by setting
  \begin{equation*}
    X_i=\left(\mathbbm{1}_i-\tfrac1{|F|}\right)\tfrac{a_{i,1}}{S}.
  \end{equation*}
  Observe that these variables are independent, and satisfy $\Exp X_i=0$ and
  $|X_i|\le 1$ and so Theorem~\ref{thm:hoeff} applied with $a=\sqrt{12t|F|/S}$
  asserts that
  \begin{equation}\label{eq:assign:P}
     \Prob\Big[\sum_{i\in[s]}X_i>\sqrt{12t|F|/S}\Big]
       \le\exp\left(-\frac{12t|F|}{S\cdot 2s}\right)
       \le\exp\left(-2|F|\right)
       <\frac1{2|F|}
  \end{equation}
  where we used $s\le3t/S$. Now, by definition we have
  \[
    X:=\sum_{i\in[s]}X_i=\frac1{S}
      \sum_{(i,1)\in\phi^{-1}(v_1)}\!\!\!\!\!\!\!\!\! a_{i,1}
      -\frac{t_1}{|F|S},
  \]
  and so, if~\eqref{eq:assign:fork} did not hold for $v_1$, then we had
  $X>\sqrt{12t|F|/S}$, which by~\eqref{eq:assign:P} occurs with probability
  less than $1/(2|F|)$.
  
  For the case $v=v_2\in V_2(F)$ we proceed similarly. Let $r'\le r$ be the
  number of prongs of the fork that contains $v_2$.
  We define indicator variables $\mathbbm{1}'_i$ for the events $\phi(i,2)=v_2$
  for $i\in[s]$ and random variables
  \begin{equation*}
    Y_i=\left(\mathbbm{1}'_i-\tfrac{r'}{|F|}\right)\tfrac{a_{i,2}}{S}.
  \end{equation*}
  with $\Exp Y_i=0$ and $|Y_i|\le 1$. The rest of the argument showing that
  $\sigma(v_2)$ occurs with probability strictly less than
  $1/(2|F|)$ is completely analogous to the case $v=v_1$ above. With 
  this we are done.
\end{proof}

As explained earlier these two previous lemmas will allow us to assign the
shrubs of a tree $T$ to edges of a reduced graph $\mathbb G$. By applying them
we will obtain a mapping $\psi$ from the vertices of $T$ to those of
$\mathbb G$ that is a homomorphism when restricted to the shrubs of $T$.
The following lemma transforms such a $\psi$ to a homomorphism $h$ from the
whole tree $T$ to $\mathbb G$ that ``almost'' coincides with $\psi$ provided the
structures of $T$ and $\mathbb G$ are ``compatible'' with
respect to $\psi$ in the sense of the following definition.

\begin{definition}[walk condition]\label{def:walk-cond}
Let $T$ be a tree and $C\subseteq V(T)$. A mapping $\psi\colon V(T)\setminus
C\rightarrow \mathbb G$ satisfies the \emph{walk condition} if for
any $x,y\in V(T)\setminus C$ such that there is a path $P_{x,y}$ from $x$ to
$y$ whose internal vertices are all in~$C$ there is a walk $\mathbb P_{x,y}$ between $\psi(x)$ and
$\psi(y)$ in $\mathbb G$ such that the length of $P_{x,y}$ and the length of 
$\mathbb P_{x,y}$ have the same parity. 
\end{definition}

\begin{lemma}\label{lem:change-assignment}
Let $T$ be a tree with maximal degree $\Delta$, let~$C$ be a cut of~$T$, and let
$\mathbb G$ be a graph on~$\nR$ vertices.
Let $\psi\colon V(T)\setminus C\rightarrow V(\mathbb G)$ be a
homomorphism that maps each shrub of $T$
corresponding to $C$ to an edge of $\mathbb G$ and that satisfies the walk
condition.
%
%
Then there is a homomorphism $h\colon V(T)\rightarrow V(\mathbb G)$ satisfying
\begin{enumerate}[label=\rm(h\arabic{*})]
  \item $|h(\neighbor_T(x))|\leq 2$ for all vertices $x\in V(T)$ and
    \label{lem:assign:1}
  \item $|\{x\in V(T)\,\colon\, h(x)\neq \psi(x)\}|\le 
    3|C|\Delta^{2\nR+1}$.  \label{lem:assign:2}
\end{enumerate}
\end{lemma}

Observe that Property~\ref{lem:assign:1} in this lemma asserts that images of
neighbours of any vertex in $T$ occupy at most two vertices in $\mathbb G$.
By assumption, this is clearly true for $\psi$ but we need to make sure that
$h$ inherits this feature. Property~\ref{lem:assign:2} on the other hand
states that $h$ and $\psi$ do not differ much. 
The assumption that $\psi$ satisfies the walk condition is essential 
for the construction of the homomorphism ~$h$.

\begin{proof}[Proof of Lemma~\ref{lem:change-assignment}]
\setcounter{fact}{0}
  We start with some definitions. Choose a non-empty shrub corresponding to $C$
  in $T$ and call it \emph{shrub $1$}. Then choose a cut-vertex $x^*_0\in C$
  adjacent to this shrub. We consider $x_0^*$ as the \emph{root} of the tree $T$.
  This naturally induces the following partial order $\prec$ on the vertices
  $V(T)$ of $T$: For vertices $x,y\in V(T)$ we have $x\prec y$ iff $y$ is a
  descendant of $x$ in the tree $T$ with root $x_0^*$. Note that $x_0^*$ is the
  unique minimal element of $\prec$ and the leaves of $T$ are its maximal
  elements. Further, for $x\in C$ set $W_x:=\{z\in V(T)\,\colon\,
  \textrm{dist}_T(x,z)\le 2\nR+1\; \&\; x\prec z\}$ and let $W=C\cup
  \bigcup_{x\in C}W_x$. Observe that the bound on the maximal degree of $T$
  implies that $|W|\le 2\Delta^{2\nR+1}|C|+|C|\leq 3\Delta^{2\nR+1}|C|$. 
  For $x\in V(T)\setminus W$, we set $h(x):=\phi(x)$. This ensures that
  Condition~\ref{lem:assign:2} is fulfilled. In addition the following fact holds
  because $\psi$ maps each shrub to an edge of $\mathbb G$.
  
  \begin{fact}\label{fac:assign:notW}
    The mapping $h$ restricted to $V(T)\setminus W$ is a homomorphism.
    For all vertices $x\in V(T)\setminus W$ all children $y$ of $x$
    that are not cut-vertices have the same $h(y)$.
  \end{fact}
  We shall extend $h$ to the set $W$. Our strategy is roughly as follows: We
  start by defining $h(x^*_0)$ for the root cut-vertex $x_0^*$ in a suitable
  way. Recall that all children of $x^*_0$ are contained in $W$. Then, we let $h$ map all
  non-cut-vertex children $y\in \neighbor_T(x_0^*)\setminus C$ of $x^*_0$ to a
  suitable neighbour of $h(x^*_0)$ in $\mathbb G$ and do the following for
  each of these $y$. Observe that $y$ is the root of some shrub, which we will
  call the \emph{shrub of $y$}. Now, $h(y)$ and $\psi(y)$ might be
  different. However, we will argue that there is a walk of even length $m\le
  2\nR$ between $h(y)$ and $\psi(y)$. Then we will define $h$ for all vertices
  $y'\in W_{x^*_0}$ contained in the shrub of $y$ and with distance at most
  $m$ from $y$. More precisely we will use the walk of length $m$ between
  $h(y)$ and $\psi(y)$ and let $h$ map all $y'$ with distance $i$ to $y$ to
  the $i$-th vertex of this walk. All vertices $z$ in the shrub of $y$ for
  which $h$ is still undefined after these steps are then mapped to
  $h(z):=\psi(z)$. Once this has been done for all $y\in
  \neighbor_T(x^*_0)\setminus C$ we proceed in the same way with the next
  cut-vertex:
  We choose a cut-vertex $x^*$ with parent $x$ for which
  $h(x)$ is already defined and proceed similarly for $x^*$ as we did for
  $x_0^*$. 
%

  We now make the procedure for the extension of $h$ on $W$ precise. Throughout
  this procedure we will assert the following property for all non-cut vertices
  $y$ of $T$ such that $h(y)$ is defined.
  \begin{equation}\label{eq:assign:invariant}
    \text{There is a path of even length in $\mathbb G$ between $h(y)$ and
      $\psi(y)$.}
  \end{equation}
  Observe that~\eqref{eq:assign:invariant} trivially holds for all $y\in
  V(T)\setminus W$.

   As explained, we start our procedure with the root $x_0^*$
  of the tree $T$. Let $x_1$ be the root of shrub $1$. By definition $x_1$ is
  adjacent to $x_0^*$. Note that, while $\psi$ is not defined on $x_0^*$ it is
  defined on $x_1$. Hence we can 	legitimately set $h(y)=\psi(x_1)$ for all
  neighbours $y\notin C$ of $x_0^*$ in $T$
  and choose $h(x_0^*)$ arbitrarily in~$\neighbor_{\mathbb
  G}(\psi(x_1))$. Observe that this is consistent
  with~\eqref{eq:assign:invariant} because for any neighbour $y\notin C$ of
  $x_0^*$ we have $h(y)=\psi(x_1)$ and $\textrm{dist}_T(y,x_1)\in\{0,2\}$. By
  assumption $\psi$ satisfies the walk condition.
  Hence there is a
  walk in $\mathbb G$ with even length  between $h(y)=\psi(x_1)$ and $\psi (y)$.
  Let $\mathbb{P}_y=v_0,v_1,\dots,v_m$ be a walk in $\mathbb G$ of minimal but
  even length with $v_0=h(y)$ and $v_m=\psi (y)$. As $\mathbb G$ has $\nR$
  vertices we have that $m\le 2\nR$. For all vertices $z\in W$ that are in the shrub of $y$ and satisfy
  $\textrm{dist}_T(y,z)=j$ for some $j\le m$, we then define $h(z):=v_j$. For
  the remaining vertices $z\in W$ in the shrub of $y$ we set $h(z):=\psi(z)$.
  Observe that this is again consistent with~\eqref{eq:assign:invariant} and 
  in conjunction with Fact~\ref{fac:assign:notW} implies the following condition
  (which we will also guarantee throughout the whole process of defining $h$).
  \begin{fact}\label{fac:assign:shrub}
    Let $x^*\in C$ and $y\notin C$ such that
    $h(x^*)$ and $h(y)$ are defined. Then the following holds:
    \begin{enumerate}[label={\rm(\roman{*})}]
      \item\label{fac:assign:shrub:2}
      All children $y'\notin C$ of $x^*$ have the same $h(y')$
      and $h(x^*)h(y')\in E(\mathbb G)$.
      \item\label{fac:assign:shrub:3}
      All children $y'\notin C$ of $y$ have the same $h(y')$
      and $h(y)h(y')\in E(\mathbb G)$.
    \end{enumerate}
  \end{fact}  
  In this way we have defined $h$ for all shrubs adjacent to the root $x_0^*$.
  
  Next we consider any vertex $x^*\in C\cap \neighbor_T(x_0^*)$ and set
  $h(x^*):=h(x_1)$, where $x_1$ is as defined above. We let $z^*$ be the parent
  of $x^*$, i.e., $z^*=x_0^*$. Then set $h(y):=h(z^*)$ for all children $y\notin C$ of $x^*$.
  This is consistent with Fact~\ref{fac:assign:shrub}.
  Afterwards we have the following situation: $x^*$ and $z^*=x_0^*$ are
  neighbouring cut-vertices and the vertex $x_1$ is a non-cut-vertex neighbour
  of $x_0^*$. 
  Let $y\in \neighbor_T(x^*)\setminus C$. Then we have $\textrm{dist}_T(x_1,y)=3$.
  Because $y$ and $x_1$ are both non-cut vertices the properties of $\psi$ imply as before
  that there is a walk in $\mathbb G$ of odd length between $\psi(x_1)$ and $\psi(y)$.
  By the walk condition and the facts that $h(x_1)=\psi(x_1)$ and
  $h(x_0^*)=h(y)$, we know that in $\mathbb G$ there is a walk 
  $\mathbb{P}_y$ of even length $m\le 2\nR$ between $h(y)$ and $\psi(y)$. This
  verifies~\eqref{eq:assign:invariant} for $y$.
%
%
  We thus can define $h$ for the vertices $z$ contained
  in the shrub of $y$ as above:
  if $\textrm{dist}_T(y,z)\le m$ then we use
  this path and set $h(z)$ according to $\textrm{dist}_T(y,z)$
  and otherwise we set $h(z):=\psi(z)$.
  With this we stay consistent with~\eqref{eq:assign:invariant} and
  Fact~\ref{fac:assign:shrub}. We then repeat the above procedure for all
  $x^*\in C\cap \neighbor_T(x_0^*)$ which implies that the next fact holds true.
  \begin{fact}\label{fac:assign:root}
  All vertices $x\in \neighbor_T(x_0^*)$ have the same $h(x)$.
  \end{fact}
  
   Now we are in  the following situation.
   
   \begin{fact}\label{fac:assign:situation}
     The mapping $h$ is defined on all shrubs adjacent to cut vertices $x^*$
     with $h(x^*)$ defined.
  Moreover,  for each cut vertex $x^*$ with $h(x^*)$ undefined that has a parent
  $z$ for which $h(z)$ is defined, then $z$ has a parent $z'$ with $h(z')$
  defined and $h(z)h(z')$ is an edge of $\mathbb G$.
   \end{fact}
   
   
   As long as $h$ is not defined for all $z\in
   V(T)$ we then repeat the following. We choose a cut vertex $x^*$ with
   $h(x^*)$ undefined that is minimal with respect to this property in $\prec$.
   Denote the parent of $x^*$ by  $z$  and let $z'$ be the parent of $z$. Then,
   by Fact~\ref{fac:assign:situation}, the mapping $h$
   has already been defined for $z'$ and $z$. Set $h(x^*):=h(z')$ and for all
   children $y\notin C$ of $x^*$ set $h(y):=h(z)$. 
   Because $h(z')h(z)$ is an edge of $\mathbb G$ by
   Fact~\ref{fac:assign:situation} this gives the
   following property for $x^*$ (which we, again, guarantee throughout the
   definition of $h$).
  
   \begin{fact}\label{fac:assign:cut}
   For all cut vertices $x^*\in C$ with $h(x^*)$ defined we have that $h(x^*)h(z)$ is an edge
    of $\mathbb G$, where $z$ is the parent of $x^*$. Moreover if 
   $x^*\notin \{x_0\}\cup (C\cap \neighbor_T(x_0^*))$, we have that $h(x^*)=h(z')$,
   where $z'$ is the parent of $z$.
  \end{fact}
  
   
   Moreover, the definition of $h(y)$ is consistent
   with~\eqref{eq:assign:invariant}, i.e.
   there is a path of even length in $\mathbb{G}$ between $h(y)$ and $\psi(y)$
   for all children $y\notin C$ of $x^*$.
   Accordingly we can again define $h$ for the vertices in the shrub of
   $y$ as before, using this path. 

   This finishes the description of the definition of $h$. It remains to verify
   that~$h$ is a homomorphism and satisfies Condition~\ref{lem:assign:1}. For
   the first part it suffices to check that for any $y\in V(T)\setminus
   \{x_0^*\}$ with parent $x$ we have $h(y)\in \neighbor_{\mathbb G}(h(x))$. 
   If $y$ is a vertex in some shrub then
   Facts~\ref{fac:assign:shrub}\ref{fac:assign:shrub:2}
   and~\ref{fac:assign:shrub}\ref{fac:assign:shrub:3} imply that $h(x)h(y)$ is
   an edge of~$\mathbb G$.
   If $y$ is a cut-vertex, on the other hand, 
   Fact~\ref{fac:assign:cut} implies that $h(x)h(y)$ is an edge of $\mathbb G$.
   So $h$ is a homomorphism. 
   
   Further, by Fact~\ref{fac:assign:shrub}\ref{fac:assign:shrub:2}
   and~\ref{fac:assign:shrub:3} we get for all vertices $x$ of $T$ that all
   children $x'\notin C$ of $x$ have the same $h(x')$. 
   By Fact~\ref{fac:assign:cut}, if $x\neq x_0^*$ then all children $x'\in
   C$ of $x$ and the parent $z$ of $x$ have the same $h(x')=h(z')$.  Together
   wit Fact~\ref{fac:assign:root}, this
   implies Property~\ref{lem:assign:1}.
\end{proof}

Now we are ready to prove Lemma~\ref{lem:valid}.

\begin{proof}[Proof of Lemma~\ref{lem:valid}]
  Given $\eps,\mu>0$ with $\eps\le \mu/10$ and $\nR\in \mathbb N$ we set
  $\alpha$, $n_0$ and an auxilliary constant $\beta>0$ such that
  \begin{equation}\label{eq:valid:const}
    \alpha\cdot (2\nR+1)=\tfrac12,
    \qquad
    \beta=\eps\mu/(500\nR^3),
    \quad\text{and}\quad
     n_0=(1500\nR/(\eps\mu))^4.  
  \end{equation}
 Let $\mathbb G$ be a graph of order~$\nR$ that has an odd connected matching
 $\mathbb M$ of size at least~$m$ or a fork system $\mathbb F$ of size at
 least~$f$ and ratio~$r$. Let $T$ be a tree satisfying the respective conditions
 of Case~\ref{lem:valid:matching} or~\ref{lem:valid:fork} and let $V_1$ and
 $V_2$ denote the two partition classes of~$T$ with $t_1=|V_1|\ge|V_2|=t_2$.
 We first construct
 an $S$-cut $C$ for $T$ with $S:=\beta n\le\eps\frac {n}{\nR}$.
 Lemma~\ref{lem:cut} asserts that there is such a cut $C$ with
  \begin{equation}\label{eq:valid:cut}
    |C|\le\frac{|V(T)|}{S}
    \le \frac{(1-\mu)2k\frac nk}{\beta n}
    \leBy{\eqref{eq:valid:const}} \frac{1000\nR^3}{\eps\mu}
    \leBy{\eqref{eq:valid:const}} \eps \frac n\nR.
  \end{equation}
  Let $T_1,\dots,T_s$ be the shrubs of $T$ corresponding to the cut $C$. 
  We now distinguish whether we are in Case~\ref{lem:valid:matching}
  or~\ref{lem:valid:fork} of the lemma. In both cases we will construct a
  mapping~$\psi$ that is a homomorphism from $T-C$ to either~$\mathbb M$
  or~$\mathbb F$ and satisfies the walk condition. After this case distinction
  the mapping~$\psi$ will serve as input for Lemma~\ref{lem:change-assignment}
  which we then use to finish this proof.
  \smallskip
  
  {\sl Case~\ref{lem:valid:matching}\,}:\, In this case we apply
  Lemma~\ref{lem:assign:matching} in order to obtain an assignment of the
  shrubs to matching edges of $\mathbb M$ as follows. Set $a_{i,j}:=|V(T_i)\cap
  V_j|$ for all $i\in[s]$, $j\in[2]$. This implies that $\sum_{i,j}
  a_{i,j}\le|V(T)| \le t=(1-\mu)2m\frac n\nR$ and, because $C$ is an $S$-cut,
  that $a_{i,1}+a_{i,2}\le S$ for all $i\in[s]$. Accordingly
  Lemma~\ref{lem:assign:matching} produces a mapping $\phi:[s]\times[2]\to
  V(\mathbb M)$ satisfying $\phi(i,1)\phi(i,2)\in\mathbb M$
  and~\eqref{eq:assign:matching}.
  
  We now use $\phi$ to construct a mapping $\psi\colon T\setminus C\to
  V(\mathbb M)$. Set
  $\psi(v):=\phi(a_{i,j})$ for all $v\in V(T_i)\cap V_j$. Note that this
  definition together with~\eqref{eq:assign:matching} gives
  \begin{equation}\label{eq:valid:m:psi}
    |\psi^{-1}(\ell)|\le\frac{t}{2m}+2S\le(1-\mu)\frac{n}{\nR}+2\beta n
  \end{equation}
  for all vertices $\ell$ of $\mathbb M$. Each edge of $T-C$ lies in some shrub
  $T_i$, $i\in[s]$ and as the mapping~$\phi$ sends each shrub $T_i$ to an edge of
  $\mathbb M$, the mapping $\psi$ is a homomorphism from $T-C$ to $\mathbb M$.
  Moreover, as $\mathbb M$ is an odd connected matching, for any pairs of
  vertices $\ell,\ell'\in V(\mathbb M)$ there is as well an even as also an odd
  walk in $\mathbb G$ between $\ell$ and $\ell'$.  Thus $\psi$ satisfies the walk
  condition.
  \smallskip
  
  {\sl Case~\ref{lem:valid:fork}\,}:\, In this case we apply
  Lemma~\ref{lem:assign:fork} in order to obtain an assignment of the shrubs
  corresponding to $C$ to edges of $\mathbb F$. For this application we use
  parameters $t_1=|V_1|$, $t_2=|V_2|$ and $a_{i,j}:=|V(T_i)\cap V_j|$ for all
  $i\in[s]$, $j\in[2]$.
  It follows
  that $\sum_{i} a_{i,1}=t_1$ and $\sum_{i} a_{i,2}=t_2$. Because~$C$ is an
  $S$-cut, we further have $a_{i,1}+a_{i,2}\le S$ for all $i\in[s]$. Accordingly
  Lemma~\ref{lem:assign:fork} produces a mapping $\phi:[s]\times[2]\to
  V(\mathbb F)$ satisfying $\phi(i,1)\phi(i,2)\in\mathbb F$
  and~\eqref{eq:assign:fork}.

  Again, we use $\phi$ to construct the mapping $\psi\colon T\setminus C\to
  V(\mathbb F)$ by setting
  $\psi(v):=\phi(a_{i,j})$ for all $v\in V(T_i)\cap V_j$. 
  By assumption we have $t_1\le t'=(1-\mu)f\frac n\nR$ and $t_2\le
  \frac{t'}r=(1-\mu)f\frac {n}{r\nR}$ and hence 
  $t_1+t_2\le(1-\mu)f\frac n\nR(1+\frac 1r)$.
  Together
  with~\eqref{eq:assign:fork} this implies for all vertices $\ell_1\in
  V_1(\mathbb F)$ and $\ell_2\in V_2(\mathbb F)$ that
  \begin{equation}\label{eq:valid:f:psi}
  \begin{split}
       |\psi^{-1}(\ell_1)|&\le\frac{(1-\mu)f\tfrac n\nR}{f}+
       \sqrt{12(1-\mu)f\tfrac n\nR(1+\tfrac 1r)Sf} \\
       &\le(1-\mu)\tfrac{n}{\nR}+2fn\sqrt{6\beta/\nR}
       \,,  
  \end{split}
  \end{equation}
and similarly
\begin{equation}\label{eq:valid:rf:psi}
  |\psi^{-1}(\ell_2)|
  \le\frac{r(1-\mu)f\tfrac{n}{r\nR}}{f}+2fn\sqrt{6\beta/\nR}
  \le(1-\mu)\frac{n}{\nR}+2fn\sqrt{6\beta/\nR}\,.  
\end{equation}
  Putting~\eqref{eq:valid:f:psi} and~\eqref{eq:valid:rf:psi} together, we
  conclude for any $\ell\in V(\mathbb F)$ that
  \begin{equation}\label{eq:valid:fork:psi}
  |\psi^{-1}(\ell)|\le (1-\mu)\tfrac{n}{\nR}+2fn\sqrt{6\beta/\nR}
  \le (1-\mu)\tfrac{n}{\nR}+2n\sqrt{6\beta\nR}\,.
  \end{equation}
  
  As before it is easy to see that the mapping $\psi$ is
  a homomorphism from $T-C$ to $\mathbb F$.
  Moreover, as $\mathbb F$ is a fork system, there is an even walk between any
  two vertices $\ell,\ell'\in V_1(\mathbb F)$ and between any two vertices
  $\ell,\ell'\in V_2(\mathbb F).$ Because $\psi$ maps vertices of $V_1(T)$ to
  $V_1(\mathbb F)$ and vertices of $V_2(T)$ to vertices of $V_2(\mathbb F)$,   
  the mapping $\psi$ also satisfies the walk condition in this case.
  \smallskip

  {\sl Applying Lemma~\ref{lem:change-assignment}\,}:\,
  In both Cases~\ref{lem:valid:matching} and~\ref{lem:valid:fork} we
  now apply Lemma~\ref{lem:change-assignment} in order to transform~$\psi$ into
  a homomorphism from the whole tree $T$ to $\mathbb G$.
  With input $T$, $\Delta:=n^\alpha$, $C$, $\mathbb G$, and~$\psi$ this
  lemma produces a homomorphism $h:V(T)\to V(\mathbb G)$
  satisfying~\ref{lem:assign:1} and~\ref{lem:assign:2}.
  We claim that~$h$ is the desired $(\mu/2,(1-\eps)\frac n\nR)$-valid
  assignment.
  
  Indeed, $h$ is a homomorphism and so we have Condition~\ref{def:valid:hom} of
  Definition~\ref{def:valid}.
  Condition~\ref{def:valid:N} follows from~\ref{lem:assign:1}. To check
  Condition~\ref{def:valid:place} let $\ell$ be any vertex of $\mathbb G$. We
  need to verify that $|h^{-1}(\ell)|\le(1-\frac12\mu)(1-\eps)\frac n\nR$.
  By~\ref{lem:assign:2} we have
  $|h^{-1}(\ell)|\le|\psi^{-1}(\ell)|+3|C|\Delta^{2\nR+1}$. Because $|C|\le
  1000\nR/(\eps\mu)$ by~\eqref{eq:valid:cut} and
  $\Delta^{2\nR+1}=n^{\alpha\cdot (2\nR+1)}=\sqrt{n}$ by~\eqref{eq:valid:const}
  we infer that
  \begin{multline*}
    |h^{-1}(\ell)|
    \le|\psi^{-1}(\ell)|+\frac{3000\nR}{\eps\mu}\sqrt{n}
    \leByRef{eq:valid:const} |\psi^{-1}(\ell)|+\beta n \\
    \leBy{\eqref{eq:valid:m:psi},\eqref{eq:valid:fork:psi}}
    (1-\mu)\tfrac{n}{\nR}
      +\max\left\{2\beta n,2n\sqrt{6\beta\nR}\right\}+\beta n 
    \leByRef{eq:valid:const} (1-\tfrac12\mu)(1-\eps)\tfrac n\nR,
  \end{multline*}
  where  in the last inequality we use that $\eps\le\mu/10$.
\end{proof}

\section{Proof of the main embedding lemma}
\label{sec:emb}

Our proof of Lemma~\ref{lem:emb} uses a greedy stragety for embedding the
vertices of a tree with valid assignment into the given host graph. 

\begin{proof}[Proof of Lemma~\ref{lem:emb}]
\setcounter{fact}{0}
  Let $V_0\dcup V_1\dcup\dots\dcup V_\nR$ be an $(\eps,d)$-regular partition of
  $G$ with reduced graph $\mathbb G$ and let $T$ be a tree
  with~$\Delta(T)\le\Delta$ and with a $(\rho,(1-\eps)\frac n\nR)$-valid
  assignment~$h$ to~$\mathbb G$.
  Further, let~$C$ be an $S$-cut of~$T$, let $T_1,\dots,T_s$ be the
  shrubs of~$T$ corresponding to $C$, and assume that  
  \begin{equation}\label{eq:emb:ass}
    (\tfrac1{10}d\rho-10\eps)\tfrac{n}{\nR}\ge|C|+S+\Delta\,.
  \end{equation}
  As last preparation we arbitrarily divide each cluster
  $V_i=V'_i\dcup V^*_i$ into a set $V'_i$ of size $(1-\frac12\rho)|V_i|$, which
  we will call \emph{embedding space}, and the set of remaining vertices
  $V^*_i$, the so-called \emph{connecting space}.
  Next we will first specify the order in which we embed the vertices of~$T$
  into~$G$, then describe the actual embedding procedure, and finally justify
  the correctness of this procedure. 
  
  Pick an arbitrary vertex $x^*_1\in C$ as 
  root of $T$ and order the cut vertices $C=\{x^*_1,\dots,x^*_c\}$, $c=|C|$ in
  such a way that on each $x^*_1-x^*_i$-path in $T$ there are no~$x^*_j$ with
  $j>i$. Similarly, for each $i\in[s]$ let $t(i)$ denote the number of vertices
  in the shrub $T_i$ and order the vertices $y_1,\dots,y_{t(i)}$ of
  $T_i$ such that all paths in $T_i$ starting at the root of $T_i$ have solely
  ascending labels. For embedding~$T$ into~$G$ we process the cut vertices and
  shrubs according to these orderings, more precisely we first
  embed~$x^*_1$, then all shrubs $T_i$ that have~$x^*_1$ as parent, one after the other. For embedding $T_i$ we embed its
  vertices in the order $y_1,\dots,y_{t(i)}$ defined above. Then we embed the
  next cut vertex~$x^*_2$ (which is a child of one of the shrubs embedded
  already or of $x^*_1$), then all child shrubs of~$x^*_2$, and so on. Let
  $x_1,\dots,x_{|V(T)|}$ be the corresponding ordering of $V(T)$.
  
  Before turning to the embedding procedure itself,
  observe that Property~\ref{def:valid:N} of Definition~\ref{def:valid} asserts
  the following fact.
  For a vertex~$x_j$ of $T$ and for $i\in[\nR]$ let $\neighbor_{i}(x_j)$ be
  the set of neighbours~$x_{j'}$ of~$x_j$ in $T$ with $j'>j$ and $h(x_{j'})=i$. 
  \begin{fact}\label{fac:emb:Ni}
     For all vertices~$x_j$ of $T$ at most two sets $\neighbor_i(x_j)$ are non-empty.
  \end{fact}
  The idea for embedding $T$ into $G$ is as follows. We equip each vertex $x\in
  V(T)$ with a \emph{candidate set} $V(x)\subset V_{h(x)}$ and from which~$x$
  will choose its image in $G$. To start with, we set $V(x^*):=V^*_{h(x^*)}$ for
  all vertices $x^*\in C$ and $V(x):=V'_{h(x)}$ for all other vertices~$x$. Cut
  vertices will be embedded to vertices in a connecting space and non-cut
  vertices to vertices in an embedding space.
  Then we will process the vertices of $T$ in the order $x_1,\dots,x_{|V(T)|}$
  defined above and embed them one by one.
  Whenever we embed a cut vertex~$x^*$ to a vertex $v$ in this procedure we
  will set up so-called \emph{reservoir sets} $R_i\subset V_i\cap\neighbor_G(v)$ for
  all (at most two) clusters~$V_i$ such that some child~$x$ of~$x^*$ is assigned
  to~$V_i$, i.e., $h(x)=i$. 
  Reservoir sets will be used for embedding the children of cut vertices.
  We (temporarily) remove the vertices in these reservoir sets from
  all other candidate sets but put them back after processing all child shrubs
  of~$x^*$. 
  This will ensure that we have enough space
  for embedding children of~$x^*$, even after possibly embedding $\Delta-1$
  child shrubs of~$x^*$.

  Now let us provide the details of the embedding procedure.
  Throughout, $x^*$ will denote the cut
  vertex whose child-shrubs are currently processed. The set $U$ will denote the
  vertices in $G$ used so far; thus initialize this set to
  $U:=\emptyset$. As indicated above, initialize $V(x^*):=V^*_{h(x^*)}$ for all
  vertices $x^*\in C$ and $V(x):=V'_{h(x)}$ for $x\in V(T)\setminus C$,
  and set $R_i:=\emptyset$ for all $i\in[\nR]$.
  For constructing an embedding $f\colon V(T)\to V(G)$ of $T$ into $G$,
  repeat the following steps:  
  \begin{enumerate}[label={\rm\arabic{*}.},ref={\rm\arabic{*}},leftmargin=*]
    \item\label{alg:emb:1}
      Pick the next vertex~$x$ from $x_1,\dots,x_{|V(T)|}$.
    \item\label{alg:emb:2}
      Choose a vertex $v\in V(x)\setminus U$ that is typical with
      respect to $V(y)\setminus U$ for all unembedded
      $y\in \neighbor_T(x)$, set $f(x)=v$, and $U:=U\cup\{v\}$.      
    \item\label{alg:emb:3}
      For all unembedded $y\in \neighbor_T(x)$ set
      $V(y):=(V(y)\setminus U)\cap \neighbor_G(v)$.
    \item\label{alg:emb:4}
      If $x\in C$ then set $x^*:=x$. 
      Further, for all $i$ with $\neighbor_i(x)\setminus C\neq\emptyset$ arbitrarily choose a
      reservoir set $R_i\subset(V'_i\setminus U)\cap \neighbor_G(v)$ of size
      $5\eps \frac nk+\Delta$, set $V(y):=R_i$ for all $y\in
      \neighbor_i(x)\setminus C$, and (temporarily) remove $R_i$ from all other
      candidate sets in $V'_i$, i.e.,
      set $V(y'):=V(y')\setminus R_i$ for all $y'\in V(T)\setminus \neighbor_i(x)$.
    \item\label{alg:emb:5}
      After the vertices of all child shrubs of~$x^*$ are
      embedded put the vertices in $R_i$
      back to all candidate sets in $V'_i$ for all $i\in[\nR]$, i.e.,
      $V(y):=V(y)\cup R_i$ for all $y\in V(T)\setminus C$ with $h(y)=i$, and set $R_i:=\emptyset$.
  \end{enumerate}
  Steps~\ref{alg:emb:3} and~\ref{alg:emb:4} of this procedure guarantee
  for each vertex $y$ with embedded parent~$x$ that the candidate set $V(y)$ is
  contained in $N_G(f(x))$. Accordingly, if we can argue that in
  Step~\ref{alg:emb:2} we can always choose an image $v$ of~$x$ in $V(x)$ (and
  that we can choose the reservoir sets in Step~\ref{alg:emb:4}) we indeed
  obtain an embedding $f$ of $T$ into $G$. 
  To show this we first collect some observations that will be usefull in the
  following analysis. The order of $V(T)$ guarantees that all child shrubs
  of a cut vertex are embedded before the next cut vertex. Notice that this
  implies the following fact (cf.\ Step~\ref{alg:emb:4} and
  Step~\ref{alg:emb:5}).
  \begin{fact}\label{fac:emb:R}
    For all $i\in[\nR]$, at
    any point in the procedure, the reservoir set $R_i$ 
    satisfies $|R_i|=5\eps\frac{n}{\nR}+\Delta$ if there is a neighbour~$x$ of
    the current cut-vertex~$x^*$ such that $h(x)=i$ and $|R_i|=0$ otherwise.
    In addition no reservoir set gets changed before all child shrubs of~$x^*$
    are embedded.
  \end{fact}
  Further, since $h$ is a $(\rho,(1-\eps)\frac n\nR)$-valid assignment and only
  cut-vertices are embedded into connecting spaces $V^*_i$, we always have
  \begin{equation}\label{eq:emb:Vi}
    |V'_i\cap U|\le(1-\tfrac12\rho)\tfrac n\nR
    \quad\text{and}\quad
    |V^*_i\cap U|\le|C| \quad\text{for all $i\in[\nR]$}\,.
  \end{equation}
Now we check that Steps~\ref{alg:emb:2} and~\ref{alg:emb:4} can always be
performed. To this end consider any iteration of the embedding procedure and
suppose we are processing vertex~$x$. We distinguish three cases.

  \smallskip  
  {\sl Case~1:}\,
  Assume that~$x$ is a cut-vertex. Then we had
  $V(x)=\smash{V^*_{h(x)}}$ until the parent~$x'$ of~$x$ got embedded.
  In the iteration when~$x'$ got embedded then the
  set $V(x)$ shrunk to a set of size at least
  $(d-\eps)\smash{|V^*_{h(x)}\setminus U|}$ in Step~\ref{alg:emb:3} because
  $f(x')$ is typical with respect to $\smash{V^*_{h(x)}\setminus U}$. 
  No vertices embedded between~$x'$
  and~$x$ (except for possible vertices in $C$) alter $V(x)$, and so
  by~\eqref{eq:emb:Vi} we have
  \begin{equation*}
    |V(x)\setminus U|\ge(d-\eps)|V^*_{h(x)}|-|C|\ge(d-\eps)\tfrac12\rho
    \tfrac n\nR-|C|\gByRef{eq:emb:ass} 4\eps \tfrac n\nR
  \end{equation*}
  when we are about to choose $f(x)$. 
  By Fact~\ref{fac:emb:Ni} at most two of the sets $\neighbor_i(x)$ are non-empty
  and each of these two sets can contain cut vertices $y^*$ and non-cut
  vertices $y'$. We clearly have $V(y^*)=V^*_i$ and $V(y')=V'_i$ and so there
  are at most $4$ different sets $V(y)\setminus U$, each of size at least
  $\frac12\rho \frac nk-|C|>\eps \frac nk$ by~\eqref{eq:emb:Vi}
  and~\eqref{eq:emb:ass}, with respect to which we need to choose a typical $f(x)$.
  By Lemma~\ref{lem:typical} there are less than $4\eps \frac nk$ vertices in
  $V(x)\setminus U$ (which is a subset of $V_i$) that do not fulfil this requirement. 
  Hence we can choose $f(x)$ whenever $x\in
  C$. In addition, we can choose the reservoir sets in Step~\ref{alg:emb:4} of
  this iteration: Indeed, let $i$ be such that $\neighbor_i(x)\setminus
  C\neq\emptyset$ and let $y\in \neighbor_i(x)\setminus C$ be a neighbour of~$x$ we
  wish to embed to $V_i$.
  In Step~\ref{alg:emb:2}, when we choose $f(x)$,
  then $V(y)=V'_i$ and so $f(x)$ is
  typical with respect to $V'_i\setminus U$. 
  By Lemma~\ref{lem:typical} and~\eqref{eq:emb:Vi} we thus have in
  Step~\ref{alg:emb:3} of this iteration that
  \begin{equation*}
    |(V'_i\setminus U)\cap \neighbor_G(v)|\ge(d-\eps)|V'_i\setminus
    U|\ge(d-\eps)\tfrac12\rho \tfrac n\nR
    \geByRef{eq:emb:ass} 5\eps \tfrac n\nR+\Delta.
  \end{equation*}
  Therefore we can choose $R_i$ in Step~\ref{alg:emb:4}.

  \smallskip  
  {\sl Case~2:}\,
  Assume that~$x$ is not in $C$ but the child of a cut
  vertex~$x^*$. Then $V(x)=R_{h(x)}$ before~$x$ gets embedded. Moreover, due to
  Step~\ref{alg:emb:4}, $R_i$ has been removed from all candidate sets besides
  those of the at most $\Delta$ neighbours of~$x^*$. By Fact~\ref{fac:emb:R} we
  have $|R_{h(x)}|=5\eps \frac nk+\Delta$ and so we conclude that $|V(x)\setminus
  U|\ge 5\eps\frac{n}{\nR}>4\eps\frac{n}{\nR}$. As in the previous case, there
  are at most four different sets $V(y)\setminus U$ for unembedded neighbours
  $y$ of~$x$, each of size at least $\frac12\rho\frac{n}{\nR}-|R_{h(y)}|
  =\frac12\rho \frac nk-5\eps \frac nk-\Delta\ge\eps \frac nk$
  by~\eqref{eq:emb:ass} and~\eqref{eq:emb:Vi}. Thus Lemma~\ref{lem:typical}
  guarantees that there is $v\in V(x)\setminus U$ which is typical with
  respect to all these sets $V(y)\setminus U$ and hence we can choose $f(x)$ in
  this case.

  \smallskip  
  {\sl Case~3:}\,
  As third and last case, let~$x$ be a vertex of some shrub $T_j$ which is the
  child of a (non-cut) vertex~$x'$ of $T_j$. Until~$x'$ got embedded we
  had $V(x)=V'_{h(x)}\setminus R_{h(x)}$ and so, $v'=f(x')$ was chosen typical
  with respect to $V'_{h(x)}\setminus(R_{h(x)}\cup U)$ where $U$ is the set of
  used vertices in $G$ at the time when~$x'$ got embedded. In the
  corresponding iteration $V(x)$ shrunk to $(V'_{h(x)}\setminus(R_{h(x)}\cup
  U))\cap \neighbor_G(v')$. This together with~\eqref{eq:emb:Vi} implies that
  immediately after this shrinking we had
  \begin{equation*}
    |V(x)\setminus U|
    \ge(d-\eps)(\tfrac12\rho \tfrac n\nR-|R_{h(x)}|)
    \ge(d-\eps)(\tfrac12\rho \tfrac n\nR-5\eps \tfrac n\nR-\Delta)
    \gByRef{eq:emb:ass} 4\eps \tfrac n\nR+|T_j|.
  \end{equation*}
  By construction only vertices from $T_j$ come between~$x'$ and~$x$ in the
  order of $V(T)$ and so when we want to embed~$x$ in the procedure above we
  still have $|V(x)\setminus U|>4\eps \frac nk$ where $U$ now is the set of
  vertices used until the embedding of~$x$. Similarly as in the other two cases there are
  at most four different types of candidate sets for non-embedded neighbours
  of~$x$, all of these have more than~$\eps \frac nk$ vertices and so
  Lemma~\ref{lem:typical} allows us to choose an $f(x)\in V(x)\setminus U$
  typical with respect to these sets. This concludes the case distinction and
  hence the proof of correctness of our embedding procedure.
\end{proof}


\section{Coloured tripartite graphs are either good or odd}
\label{sec:comb}

\subsection{Some tools}

In this section we collect some simple but useful propositions. We start with
two observations about matchings in $\eta$-complete graphs. The first one
states that a bipartite $\eta$-complete coloured graph contains a reasonably
big matching in one of the two colours.

\begin{proposition}\label{prop:match}
  Let $K$ be a coloured graph on $n$ vertices and let $D$
  and $D'$ be vertex sets of size at least $m$ in $K$.
  If $K[D,D']$ is $\eta$-complete then $K[D,D']$ contains a matching $M$ either
  in red or in green of size at least $\frac{m}2-\eta n$.
\end{proposition}
\begin{proof}
  Assume without loss of generality that $|D|\le|D'|$. Colour a vertex $v\in D$
  with red if it has more red-neighbours than green-neighbours in $K[D,D']$ and
  with green otherwise. By the pigeon-hole principle there is a set $X\subset D$
  of size $\frac12|D|$ such that all vertices in $X$ have the same
  colour, say red. But then each vertex in $X$ has at least $\frac12|D'|-\eta
  n\ge|X|-\eta n$ red-neighbours in $D'$. Accordingly we can greedily construct a
  red matching of size at least $|X|-\eta n\ge\frac{m}{2}-\eta n$ between $X$ and
  $D'$.
\end{proof}

The next proposition gives a sufficient condition for the existence of an almost
perfect matching in a subgraph of $K\in\K{\eta}$.

\begin{proposition}\label{prop:perfect}
  Let $K\in\K{\eta}$ have partition classes $A$, $B$, and $C$ and
  let $A'\subset A$, $B'\subset B$, $C'\subset C$ with
  $|A'|\ge|B'|\ge|C'|$. If $|A'|\le|B'\cup C'|$ then there is a
  matching in $K[A',B',C']$ covering at least $|A'\cup B'\cup C'|-4\eta n-1$
  vertices.
\end{proposition}
\begin{proof}
  Let $x:=|B'|-|C'|$ and $y:=\lfloor\frac12(|A'|-x)\rfloor$.
  Observe that $x\le|B'|\le|A'|$. Hence $y\ge 0$, 
  $x+y\le\frac12(|A'|+x)\le\frac12(|B'\cup C'|+x)=|B'|$, and 
  $y\le\frac12(|A'|-x)\le\frac12(|B'\cup C'|-x)=|C'|$. Choose arbitrary subsets
  $U_B\subset B'$ of size $x+y$, $U_C\subset C'$ of size $y$, set
  $U:=U_B\cup U_C$, $W:=B'\setminus U_B$ and $W':=C'\setminus U_C$. Clearly
  $|W'|=|C'|-y=|B'|-(x+y)=|W|$ and $|A'|-1\le x+2y=|U|\le|A'|$. Thus we can
  choose a subset $U'$ of $A'$ of size $|U|$ that covers all but at most 1
  vertex of $A'$ and so that $K[U,U']$ and $K[W,W']$ are $\eta$-complete
  balanced bipartite subgraphs. A simple greedy algorithm allows us then to
  find matchings of size at least $|U|-\eta n$ and $|W|-\eta n$ in $K[U,U']$ and
  $K[W,W']$, respectively. These matchings together form a matching in
  $K[A',B',C']$ covering at least $|U\cup U'\cup W\cup W'|-4\eta n\ge|A'\cup
  B'\cup C'|-4\eta n-1$ vertices.
\end{proof}

The following proposition shows that induced subgraphs of $\eta$-complete
tripartite graphs are connected provided that they are not too small. Moreover,
subgraphs that substantially intersect all three partition classes contain a
triangle.

\begin{proposition}\label{prop:tool}
  Let $K\in\K{\eta}$ be a graph with partition classes $A$, $B$, $C$, and let 
  $A'\subset A$, $B'\subset B$, $C'\subset C$.
  \begin{enumerate}[label={\rm(\alph{*})},leftmargin=*]
    \item\label{prop:cherry}
      If $|A'|>2\eta n$ then every pair of vertices in $B'\cup C'$
      has a common neighbour in $A'$.
    \item\label{prop:conn}
      If $|A'|,|B'|>2\eta n$ then $K[A',B']$ is connected.    
    \item\label{prop:triangle}
      If $|A'|,|B'|,|C'|>2\eta n$ then $K[A',B',C']$ contains a triangle.
  \end{enumerate}
\end{proposition}
\begin{proof}
  As $K\in\K{\eta}$, each vertex in $B'\cup C'$ is adjacent to at least
  $|A'|-\eta n>|A'|/2$ vertices in $A'$. Thus every pair of vertices in $B'$ has a
  common neighbour in $A'$ which gives~\ref{prop:cherry}.
  For the proof of~\ref{prop:conn} observe that by~\ref{prop:cherry} every pair
  of vertices in $B'$ has a common neighbour in $A'$. Since the same holds for
  pairs of vertices in $A'$ the graph $K[A',B']$ is connected.
  To see~\ref{prop:triangle} we use~\ref{prop:cherry} again and infer that every
  pair of vertices in $A'\times B'$ has a common neighbour in $C'$. As
  $|A'|,|B'|>2\eta n$ there is some edge in $A'\times B'$ and thus there is a
  triangle in $K[A',B',C']$.
\end{proof}

Similar in spirit to~\ref{prop:triangle} of Proposition~\ref{prop:tool} we can
enforce a copy of a cycle of length 5 in a system of $\eta$-complete graphs as
we show in the next proposition.

\begin{proposition}\label{prop:C5}
  Let $K$ be a coloured graph on $n$ vertices, let $c$ be a colour, $vw$ be a
  $c$-coloured edge of $K$, and let $D_1,D_2,D_3\subset V(K)$ such that all
  graphs $K[v,D_1]$, $K[D_1,D_2]$, $K[D_2,D_3]$, and $K[D_3,w]$ are $(\eta,c)$-complete bipartite graphs. Set
  $D:=\bigcup_{i\in[3]}D_i\cup\{v,w\}$. If $|D_i|>2\eta n+2$ for all $i\in[3]$
  then $K[D]$ contains a $c$-coloured copy of~$C_5$.
\end{proposition}
\begin{proof}
 By Proposition~\ref{prop:tool}\ref{prop:cherry} every pair of vertices in
 $D_1\cup D_3$ is connected by a path of colour~$c$ and length~$2$ with
 center in $D_2\setminus \{v,w\}$. Moreover, $v$ has at least $|D_1|-\eta
 n\ge1$ neighbours in $D_1$ and similarly $w$ has a neighbour in $D_3$. Hence
 there is a $c$-coloured~$C_5$ in $K[D]$.
\end{proof}



\subsection{Non-extremal configurations}\label{sec:Nextr-conf}

In the proof of Lemma~\ref{lem:odd} we will use that 
coloured graphs~$K$ from~$\K{\eta}$ have the following property~$P$.
Either one colour of~$K$ has a big odd connected matching or both colours
have big connected matchings whose components are bipartite. 
Analysing these bipartite configurations will then
lead to a proof of Lemma~\ref{lem:odd}.
Property~$P$ is a consequence of the next lemma, Lemma~\ref{lem:improve},
which states that if all connected matchings in a colour of $K$ 
are small then the other colour has bigger odd connected matchings.

\begin{lemma}[improving lemma]
\label{lem:improve}
  For every~$\eta'>0$ there are~$\eta>0$ and $n_0\in \mathbb N$ such that for
  all $n\ge n_0$ the following holds. Suppose that a coloured graph $K\in\K{\eta}$
  is neither $\eta'$-extremal nor $\frac34(1-\eta')n$-odd. Let~$M$ be a maximum connected
  matching in $K$ of colour $c$. If $\eta'n<|M|<(1-\eta')n$ then $K$
  has an odd connected matching $M'$ in the other colour satisfying $|M'|>|M|$.
\end{lemma}
\begin{proof}
\setcounter{fact}{0}
  Given~$\eta'$ define~$\tilde\eta:=\eta'/3$ and let~$\eta$ be small
  enough and $n_0$ large enough such that 
  $(\frac1{100}\eta'-5\eta)n_0>1$ (and hence $\eta<\frac1{500}\eta'$).
  For $n\ge n_0$ let~$K=(A\dcup B\dcup C,E)$ be a coloured graph
  from~$\K{\eta}$ with partition classes~$A$, $B$, and~$C$ that is neither
  $\eta'$-extremal nor $(1-\eta')3n/4$-odd. Suppose $c=\text{green}$ and hence
  that~$K$ has a maximum green connected matching~$M$ with
  $\eta'n<|M|<(1-\eta')n$. For $D,D'\in\{A,B,C\}$ with $D\neq D'$ let
  $M_{DD'}:=M\cap(D\times D')$.
  We call the $M_{DD'}$ the \emph{blocks} of $M$ and say that a
  block $M_{DD'}$ is \emph{substantial} if $|M_{DD'}|\ge\tilde\eta n$.
  Let $R$ be the set of vertices in $K$ not covered by $M$. For
  $D\in\{A,B,C\}$ let $R_D:= R\cap D$.

  \begin{fact}\label{fac:imp:1}
     We have $|R_A|-|M_{BC}|=|R_B|-|M_{CA}|=|R_C|-|M_{AB}|>\eta'n$.
  \end{fact}

  Indeed, $|R_A|+|M_{AB}|+|M_{AC}|=|R_B|+|M_{AB}|+|M_{BC}|$ and hence
  $|R_A|-|M_{BC}|=|R_B|-|M_{AC}|=|R_B|-|M_{CA}|$ which proves the first part of
  this fact. For the second part observe that
  $|R_A|+2|M_{AB}|+|R_B|+2|M_{BC}|+|R_C|+2|M_{CA}|=3n$. Hence we
  conclude from $|M|=|M_{AB}|+|M_{BC}|+|M_{CA}|<(1-\eta')n$ that
  \begin{equation*}\begin{split}
    3(|R_A|-|M_{BC}|)
    &= (|R_A|-|M_{BC}|) + (|R_B|-|M_{CA}|) + (|R_C|-|M_{AB}|) \\
    &=3(n-|M_{AB}|-|M_{BC}|-|M_{CA}|)
    >3\eta'n.
  \end{split}\end{equation*}
  This finished the proof of Fact~\ref{fac:imp:1}.

  In the remainder we assume without loss of generality that
  $|R_A|\ge|R_B|\ge|R_C|$. By Fact~\ref{fac:imp:1} this implies that
  $|M_{BC}|\ge\frac13\eta'n$ since $|M|>\eta'n$ and hence $M_{BC}$ is
  substantial.
  Our next main goal is to find a connected matching in red that is bigger than~$M$.
  For achieving this goal the following fact about red connections between
  vertices of $R$ will turn out useful.
  
  \begin{fact}\label{fac:imp:2}
    There is a vertex $u^*\in R_A$ such that $R-u^*$ is red connected.
  \end{fact}

  To see this, assume first that there is a vertex $u^*\in R_A$ that has more
  than $4\eta n$ green-neighbours in $M_{BC}$. Then more than $2\eta n$ of these
  neighbours are in, say, $M_{BC}\cap B$. Call this set of
  vertices~$B^*$. Now let $u\neq u^*$ be any vertex in $R\setminus C$. By the
  maximality of $M$ the vertex $u$ has no green-neighbours in $M(B^*)$. This
  implies that $u$ has at least $|M(B^*)|-\eta n>|M(B^*)|/2$ red-neighbours in
  $M(B^*)$. Thus any two vertices in $R\setminus C$ have a common red-neighbour
  in $M(B^*)$. A vertex $u\in R_C$ on the other hand has at least $|R_A|-\eta
  n\ge|M_{BC}|+\eta'n-\eta n>2\eta n+1$ neighbours in $R_A$ where the first
  inequality follows from Fact~\ref{fac:imp:1}. If at least $2$ of these
  neighbours are red then $u$ is red connected to $R_A-u^*$. Otherwise $u$ has
  a set $U$ of more than $2\eta n$ green-neighbours in $R_A-u^*$. But then, by
  the maximality of $M$, the graph $K[U,M_{BC}\cap B]$ is red. Since
  $|M_{BC}|\ge\eta'n>\eta n$ the vertex $u$ has a neighbour $v$ in $M_{BC}\cap
  B$. Since $u$ has a green-neighbour in $R_A$ it follows from the maximality
  of $M$ that $uv$ is red. Thus $u$ is red connected to $U$ and therefore
  to all vertices of $(R\setminus C)-u^*$.
  
  If there is no vertex in $R_A$ with more than $4\eta n$ green-neighbours in
  $M_{BC}$ on the other hand, then any two vertices in $R_A$ obviously have at
  least $|M_{BC}|-4\eta n-2\eta n\ge\frac13\eta'n-6\eta n>0$ common
  red-neighbours in $B\cap M_{BC}$. Moreover, by the maximality of~$M$, each
  vertex $v\in R_C\cup R_B$ is either red connected to $R_A$ or it has only
  red-neighbours in $M_{BC}$. Thus $v$ has a common red-neighbour with any
  vertex in $R_A$ which proves Fact~\ref{fac:imp:2} also in this case.

  \begin{fact}\label{fac:imp:3}
    $K$ has a red connected  matching $M'$ with
    $|M'|\ge|M|+\frac14\eta'n$.
  \end{fact}

  Let $uv$ be an arbitrary edge in $M_{BC}$. Then, by the maximality of $M$, one
  vertex of this edge, say $u$, has at most one green-neighbour in $R_A$. By
  Fact~\ref{fac:imp:1} we have $|R_A|\ge|M_{BC}|+\eta'n$ and since $u$ has at
  most $\eta n<\eta'n$ non-neighbours in~$R_A$ it follows that $u$ has at least
  $|M_{BC}|+1$ red-neighbours in~$R_A$. Thus, a simple greedy method allows us
  to construct a red matching $M'_{BC}$ of size $|M_{BC}|$ between 
  $R_A-u^*$ and such vertices
  $u$ of matching edges in $M_{BC}$. Let $R'_A$ be
  the set of vertices in~$R_A$ not covered by $M'_{BC}$. We repeat this process
  with $M_{AC}$ and $M_{AB}$, 
  respectively, to obtain red matchings $M'_{AC}$ and $M'_{AB}$ and sets $R'_B$
  and $R'_C$. 
  
  By maximality of $M$, for each vertex $w\in R'_A$ the
  following is true: either $w$ has no green-neighbour in $M_{BC}$, or $w$ has
  no green-neighbour in $R'_B$. Moreover~$w$ has at most $\eta n$
  non-neighbours. Observe that $|R'_B|,|R'_A|>\eta'n$ by Fact~\ref{fac:imp:1}
  and the set $X$ of vertices in $M_{BC}$ that are not covered by $M'_{BC}$
  has size at least~$\frac13\eta'n$ since $|M_{BC}|=|M'_{BC}|\ge\frac13\eta'n$
  and each edge of $M'_{BC}$ uses exactly one vertex from $M_{BC}$. This
  implies that we can again use a greedy method to construct a red matching
  $M'_R$ with edges from $(R'_A-u^*)\times(R'_B\cup X)$ of size at least
  $\frac13\eta'n-\eta n-1\ge\frac14\eta'n$. Hence we obtain a red matching
  $M':=M'_{BC}\dcup M'_{CA}\dcup M'_{AB}\dcup M'_R$ of size at least
  $|M|+\frac14\eta'n$. For establishing
  Fact~\ref{fac:imp:3} it remains to show that $M'$ is red connected.
  This follows from Fact~\ref{fac:imp:2} since each edge of $M'$ intersects
  $R-u^*$.
  
  If the matching $M'$ is odd then the proof of Lemma~\ref{lem:improve} is
  complete. Hence assume in the remainder that $M'$ is even. 
  Since $M'$ intersects $R-u^*$ this together with Fact~\ref{fac:imp:2}
  immediately implies the next fact. 
  For simplifying the statement as well as the following arguments we will first
  delete the vertex $u^*$ from $K$ (and let $K$ denote the resulting graph from
  now on).

  \begin{fact}\label{fac:imp:odd}
    No odd red cycle in $K$ contains a vertex of $R$.
  \end{fact}

  Fact~\ref{fac:imp:spider} below
  uses this observation to conclude that  $K$ is extremal, contradicting the
  hypothesis of Lemma~\ref{lem:improve}. To prepare the proof of this fact we first
  need some auxiliary observations.

  \begin{fact}\label{fac:imp:5}
    For $\{D,D',D''\}=\{A,B,C\}$, if $M_{DD'}$ is a substantial block then
    there is a vertex $v^*\in R_{D''}$ such that
    $K[M_{DD'},R_{D''}-v^*]$ is red and $K[M_{DD'}]$ is
    green.
  \end{fact}

  We first establish the first part of the statement. We may assume that there
  are vertices $v^*\in R_{D''}$ and $v\in M_{DD'}$ such that $v^*v$ is green (otherwise we
  are done). Without loss of generality $v\in D$.  
  Let $X=\neighbor(v^*)$.
  Then, by the maximality of~$M$, all edges between $v^*$ and $X\cap R$ are
  red. By Fact~\ref{fac:imp:odd} this implies that all edges between $X\cap R_D$
  and $X\cap R_{D'}$ are green. Since $\min\{|X\cap R_D|,|X\cap R_{D'}|\}>\eta
  n$, this set of edges is not empty. We use the maximality of~$M$ to infer that
  all edges between $M_{DD'}$ and $X\cap(R_D\cup R_{D'})$ are red. Using Fact~\ref{fac:imp:odd} this in turn implies that edges between
  $Y:=M_{DD'}\cap X$ and $v^*$ are green. By the maximality of $M$ all edges
  between $M(Y)$ and $R_{D''}-v^*$ are consequently red. We claim that therefore
  $K[R_{D'}\cap X,R_{D''}-v^*]$ 
  is green. Indeed, assume there was a red edge $ww'\in
  R_{D'}\cap X\times (R_{D''}-v^*)$. Then $w$ and $w'$ have at least
  $|M(Y)\cap D|-2\eta n\ge|M_{DD'}|-3\eta n\ge\tilde\eta n-3\eta n>0$ common
  neighbours $w''$ in $M(Y)\cap D$. Since edges between $M(Y)$ and $R_{D''}-v^*$ and
  edges between $M_{DD'}$ and $X\cap R_{D'}$ are red, so are the edges $ww''$ and
  $w'w''$ and thus we have a red triangle $ww'w''$ contradicting
  Fact~\ref{fac:imp:odd}. By Fact~\ref{fac:imp:1} we have $|R_{D'}\cap
  X|\ge\eta'n-\eta n>\eta n$ and so each vertex in $R_{D''}-v^*$ is
  connected by a green edge to some vertex in $R_{D'}\cap X$. The
  maximality of $M$ implies that $K[M_{DD'},R_{D''}-v^*]$ is red as
  required. For the second part of Fact~\ref{fac:imp:5} observe that the fact
  that $K[M_{DD'},R_{D''}-v^*]$ is red and $|R_{D''}|\ge\eta'n>2\eta n+1$
  imply that each pair of vertices in $M_{DD'}$ has a common red neighbour in
  $R_{D''}-v^*$ and so by Fact~\ref{fac:imp:odd} the graph $K[M_{DD'}]$ is
  green. This establishes Fact~\ref{fac:imp:5}.
  
  Now we also delete all (at most 3) vertices from $R$ that play the r\^ole of
  $v^*$ in Fact~\ref{fac:imp:5} (and again keep the names for the resulting
  sets).

  \begin{fact}\label{fac:imp:6}
    Suppose that $\{D,D',D''\}=\{A,B,C\}$ and that  $M_{DD'}$ is a
    substantial block. Then
    for one of the sets $D$ and $D'$, say for $D$, the graph $K[M_{DD'},R_D]$ is
    red and $K[R_{D''},R_D]$ is green. For the other
    set $D'$ the following is true. If $v\in R_{D'}$ then $K[v,M_{DD'}]$ and
    $K[v,R]$ are monochromatic, with distinct colours.
  \end{fact}
  
  We start with the first part of this fact and distinguish two cases.
  First, assume that there is a red edge $ww'$ with $w\in R_{D''}$ and
  $w'\in R_{D'}$.
  We will show that in this case  $K[M_{DD'},R_D]$ is red and
  $K[R_{D''},R_D]$ is green. Since $M_{DD'}$ is
  substantial, edges between $w$ and $M_{DD'}$ are red by
  Fact~\ref{fac:imp:5} and hence, owing to Fact~\ref{fac:imp:odd}, edges between
  $M_{DD'}\cap \neighbor(w)$ and $w'$ are green. Since $K[M_{DD'}]$ is green by
  Fact~\ref{fac:imp:5}, since $M$ is maximal, and since each vertex in
  $M_{DD'}\cap D'$ has some neighbour in $M_{DD'}\cap \neighbor(w')$ this
  implies that all edges between $M_{DD'}$ and $R_{D}$ are red. Moreover, edges between
  $M_{DD'}\cap D'$ and $R_{D''}$ are red by Fact~\ref{fac:imp:5} and hence we
  conclude from Fact~\ref{fac:imp:odd}
  that $K[R_{D''},R_D]$ is green 
  If, on the other hand, there is no red edge between $R_{D''}$ and $R_{D'}$
  then the first part of the fact is true with $D$ and $D'$ interchanged:
  Clearly $K[R_{D''},R_{D'}]$ is green and by maximality of $M$
  we infer that $K[M_{DD'},R_{D'}]$ is red.
  
  For the second part of the fact suppose that
  $K[M_{DD'},R_D]$ is red and $K[R_{D''},R_D]$ is green. Let
  $v\in R_{D'}$ and assume first that $v$ has a green neighbour in
  $M_{DD'}$. The maximality of $M$ then implies that 
  $K[v,R]$ is red and since
  $K[R_{D''},M_{DD'}]$ is also red (by Fact~\ref{fac:imp:5}) we get that
  $K[v,M_{DD'}]$ is green. 
  Hence it remains to consider the case that $K[v,M_{DD'}]$ is red.
  By Fact~\ref{fac:imp:5} the graph $K[R_{D''},M_{DD'}]$ is red and so
  Fact~\ref{fac:imp:odd} forces the graph $K[v,R_{D''}]$ to be green. To
  show that also $K[v,R_D]$ is green assume to the contrary that there
  is a red edge $vw$ with $w\in R_D$. Recall that $K[v,M_{DD'}\cap D]$,
  $K[M_{DD'}\cap D,R_{D''}]$, $K[R_{D''},M_{DD'}\cap D']$, and $K[M_{DD'}\cap
  D',w]$ are red (and clearly $\eta$-complete). Since $|M_{DD'}\cap D|,|R_{D''}|,|M_{DD'}\cap
  D'|\ge\tilde\eta n-1\ge2\eta n+2$ we can apply Proposition~\ref{prop:C5} to
  infer that there is a red $C_5$ touching $R$ which contradicts Fact~\ref{fac:imp:odd}.

  \begin{fact}\label{fac:imp:7}
    If $M_{DD'}$ and $M_{D'D''}$ are substantial, then
    $K[M_{DD'},M_{D'D''}]$ and $K[R_{D''},R_D]$ are green and $K[M_{DD'}\cup
    M_{D'D''}, R_{D''}\cup R_D]$ is red.
    Moreover, if $v\in R_{D'}$ then $K[v,M_{DD'}\cup M_{D'D''}]$ and $K[v,R]$
    are monochromatic, with distinct colours.
  \end{fact}
  By Fact~\ref{fac:imp:6} every vertex in $R_{D''}\cup R_D$ sends some green
  edges to $R$ and hence the maximality of~$M$ implies that $K[M_{DD'}\cup M_{D''D'},
  R_{D''}\cup R_D]$ is red. Since there is no red triangle touching~$R$, the
  graphs $K[M_{DD'}\cap D,M_{D''D'}\cap D']$, $K[M_{DD'}\cap D',M_{D''D'}\cap
  D'']$, and $K[R_{D''},R_D]$ are green. Using Proposition~\ref{prop:C5} we
  get similarly as before that also edges in $K[M_{DD'}\cap D,M_{D'D''}\cap
  D'']$ are green, since otherwise there was a red $C_5$ touching $R$. It
  remains to show the second part of Fact~\ref{fac:imp:7}.
  By Fact~\ref{fac:imp:6} the
  graph $K[v,R]$ is monochromatic. Moreover, applying Fact~\ref{fac:imp:6}
  once to $M_{DD'}$ and once to $M_{D''D'}$, we obtain that $K[v,R]$ and
  $K[v,M_{DD'}\cup M_{D''D'}]$ are monochromatic graphs of distinct
  colours.

  Now we have gathered enough structural information to show that $K$ is
  extremal.

  \begin{fact}\label{fac:imp:spider}
    $K$ is in spider configuration with parameter
    $\tilde\eta$.
  \end{fact}

  We first argue that we can assume without loss of
  generality that 
  \[
    \text{$C$ always plays the r\^ole of $D'$ in Fact~\ref{fac:imp:6}.}    
    \eqno{(*)}
  \]  
  Indeed, by Fact~\ref{fac:imp:7} this is the case if, besides $M_{BC}$, the
  block $M_{AC}$ is substantial. If $M_{AC}$ (and hence also $M_{AB}$)  is
  not substantial on the other hand then it might be the case that $B$ plays the
  r\^ole of $D'$ in Fact~\ref{fac:imp:6}. Then however we may delete at
  most $\tilde\eta n$ vertices from $R_B$ in order to guarantee $|R_B|\le|R_C|$
  and then the following argument still works with $B$ and $C$ interchanged.
  
  To obtain the spider configuration set $A_1:=R_A$, $B_1:=R_B$, let $C_1$ be
  the set of those vertices $v\in R_C$ such that $K[v,M_{BC}]$ is red, let
  $C_C:=R_C\setminus C_1$, and define $D_{D'}:=M_{DD'}\cap D$ for all
  $D,D'\in\{A,B,C\}$ with $D\neq D'$. If any of the sets we just
  defined has less than $\tilde\eta n$ vertices delete all vertices in this
  set. Finally, define $A_2$, $B_2$, $C_2$ as in the definition of the
  spider configuration (Definition~\ref{def:ext}).
  Observe that this together with Fact~\ref{fac:imp:6} implies that 
  $K[C_C,M_{BC}]$ is green and $K[C_C,R]$ is red.
  
  Now let $\{X,Y,Z\}=\{A,B,C\}$ arbitrarily. Clearly we have $|X_1\cup
  X_2|\ge(1-3\tilde\eta)n\ge(1-\eta')n$. Moreover $K[X_1,Y_2]$ is
  $\eta$-complete. We next verify that this graph is also red. We distinguish two cases. 
  First assume that $Y\neq C$. In this case
  $X_1\subset R_X$ and $Y_2=Y_{X}\cup Y_{Z}\subset
  (M_{XY}\cap Y)\cup(M_{YZ}\cap Y)$. 
  We have $Y_{Z}\neq\emptyset$ only if $M_{YZ}$ is substantial and
  then Fact~\ref{fac:imp:5} implies that $K[R_X,M_{YZ}]$ is red.
  Similarly $Y_{X}\neq\emptyset$ only if $M_{XY}$ is substantial.
  By $(*)$ Fact~\ref{fac:imp:6} implies that then $K[R_X,M_{XY}]$ is red if
  $X\neq C$. By the definition of $C_1$ we also get that $K[X_1,M_{XY}]$
  is red if $X=C$. Thus all edges between $X_1$ and $Y_2$ are red as
  desired. 
  If $Y=C$ on the other hand then $X_1\subset R_X$ and 
  $Y_2=Y_X\cup Y_{Z}\cup C_C\subset
  (M_{XY}\cap Y)\cup(M_{XZ}\cap Y)\cup C_C$. Analogous to the argument in
  the first case the graphs $K[R_X,M_{YZ}]$ and $K[R_X,M_{XY}]$ are red
  (since $X\neq C$). As noted above in addition all edges between $R$ and $C_C$
  are red and so $K[X_1,Y_2]$ is also red in this case.
  
  %
  %
  

  We finish the proof of Fact~\ref{fac:imp:spider} (and hence
  Lemma~\ref{lem:improve}) by checking that we have a spider configuration with
  colour $c=\text{red}$. Observe that the graph $K[A_1\cup B_1\cup C_1,A_2\cup
  B_2\cup C_2]$ is connected and bipartite. We now verify Conditions~\ref{def:spider:1}--\ref{def:spider:4} of the spider
  configuration.
  For Condition~\ref{def:spider:2} assume that $C_C\neq\emptyset$.
  Fact~\ref{fac:imp:7} and the definition of $C_2$ imply then that $M_{AB}$ is
  not substantial and hence $|A_B|=0$. 
  Moreover, since $|R_A|\ge|R_B|\ge|R_C|$ we get the first part of 
  Condition~\ref{def:spider:1}, and $|D'_D|=|D_{D'}|$
  is clearly true by definition. By Fact~\ref{fac:imp:1} we have
  $n-|M_{D''D}\cup M_{D''D'}|=|R_{D''}|>|M_{DD'}|$ which implies
  $n-|D''_2|>|D_{D'}|$ unless $D''=C$ and $C_C\neq\emptyset$ (if $D''\not=C$
  or $C_C=\emptyset$ then $|M_{DD'}|=|D_{D'}|$). And if
  $C_C\neq\emptyset$ Condition~\ref{def:spider:2} implies $|D_{D'}|=|A_B|=0$
  and thus we also get $n-|D''_2|>|D_{D'}|$ in this case. This establishes
  Condition~\ref{def:spider:1}. 
  To see Condition~\ref{def:spider:3}, note that
  if $A_2$ is non-empty then either $M_{AB}$ or $M_{AC}$ are substantial.
  Since in addition $M_{BC}$ is substantial by assumption we conclude
  from Fact~\ref{fac:imp:7} that there is a green triangle connected to
  $M_{BC}$ and hence to the green matching $M$. As $K$ is not
  $\frac34(1-\eta')n$-odd this implies $\frac12|A_B\cup A_C\cup B_A\cup B_C\cup
  C_A\cup C_B|\le|M|<\frac34(1-\eta')n$.
  It remains to verify Condition~\ref{def:spider:4}.
  Assume, for a contradiction, that $C_1\neq\emptyset$ and $|A_1\cup B_1\cup
  C_1|\ge (1-\eta)\frac32n$ and $|B_1\cup C_1|>(1-\eta)\frac34n$. As
  $|R_A|\ge|R_B|\ge|R_C|\ge|C_1|$ and $C_1\neq\emptyset$ all these sets have
  size at least $\tilde\eta n$ and so $A_1=R_A$, $B_1=R_B$ and $C_1\subset R_C$.
  By Fact~\ref{fac:imp:6} and the definition of $C_1$ the graph $K[A_1,B_1,C_1]$ is
  $(\eta,\text{green})$-complete and thus contains a green triangle by
   Proposition~\ref{prop:tool}\ref{prop:triangle} and is connected by
  \ref{prop:conn} of the same proposition. Observe that this implies that any
  matching in $K[A_1,B_1,C_1]$ is connected and odd. We will show that $K[A_1,B_1,C_1]$
  contains a green matching of size at least $\frac34(1-\eta')n$ contradicting
  the fact that $K$ is not $\frac34(1-\eta')n$-odd. We distinguish two cases. If
  $|A_1|\ge|B_1\cup C_1|$ an easy greedy algorithm guarantees a green matching
  of size $|B_1\cup C_1|-\eta n>(1-3\eta)\frac34n\ge\frac34(1-\eta')n$ in
  $K[A_1,B_1\cup C_1]$. If $|A_1|\le|B_1\cup C_1|$ on the other hand there is a
  green matching covering at least $|A_1\cup B_1\cup C_1|-4\eta
  n-1>(1-4\eta)\frac32n-1\ge\frac32(1-\eta')n$ vertices in $K[A_1,B_1,C_1]$
  by Proposition~\ref{prop:perfect}.
\end{proof}

We will now use Lemma~\ref{lem:improve} to prove Lemma~\ref{lem:odd}.

\begin{proof}[Proof of Lemma~\ref{lem:odd}]
\setcounter{fact}{0}
  Let $\eta'$ be given and set $\tilde \eta:=\eta'/15$.
  Let~$\eta_{\subref{lem:improve}}$ and $n_0$ be provided by
  Lemma~\ref{lem:improve} for input
  $\eta'_{\subref{lem:improve}}=\tilde\eta$ and set
  $\eta:=\min\{\eta_{\subref{lem:improve}},\tilde{\eta}/5\}$. Let~$K=(A\dcup
  B\dcup C,E)$ be a non-extremal coloured member of~$\K{\eta}$ with partition
  classes and assume for a contradiction that~$K$ is not $(1-\eta')3n/4$-odd.
  
  Our first step is to show that~$K$ has big green and red
  connected matchings.
  \begin{fact}\label{fac:odd:1}
    $K$ has even connected matchings~$M_r$ and~$M_g$ in red and green,
    respectively, with $|M_r|,|M_b|\ge(1-\tilde{\eta})n$.
  \end{fact}
  Assume for a contradiction that a maximum matching~$M$ in red has size less
  than $(1-\tilde{\eta})n$. By Lemma~\ref{lem:improve} applied with
  $\tilde{\eta}$ we conclude that there is an odd connected matching~$M'$ with
  $|M'|>|M|$. On the other hand~$K$ is not $\big((1-\eta')3n/4\big)$-good, 
  hence $|M'|<(1-\eta')3n/4$. Another application of
  Lemma~\ref{lem:improve} with $\tilde{\eta}\le\eta'$ thus provides us with a
  red connected matching of size bigger than~$|M'|$ which contradicts the
  maximality of~$M$. We conclude that there is a red connected matching~$M_r$,
  and by symmetry also a green connected matching~$M_g$, of size at least
  $(1-\tilde{\eta})n$. Clearly, $M_r$ and $M_g$ are even since~$K$ is not
  $\big((1-\eta')3n/4\big)$-good.

  Let~$R$ be the component of~$M_r$ and~$G$ be the component of~$M_g$
  in~$K$. Fact~\ref{fac:odd:1} states, that $R$ and $G$ are bipartite. We
  observe in the following fact that both $R$ and~$G$ substantially intersect
  all three partition classes. For this purpose define $D_r:=D\cap V(R)$ and
  $D_g:=D\cap V(G)$, and further $\bar{D}_r:=D\setminus D_r$ and
  $\bar{D}_g:=D\setminus D_g$ for all $D\in\{A,B,C\}$.
  \begin{fact}\label{fac:odd:2}
    For all $D\in\{A,B,C\}$ and $c\in\{r,g\}$ we have $|D_c|\ge2\tilde\eta n$.
  \end{fact}
  Indeed, assume without loss of generality, that $|A_r|<2\tilde\eta n$ which
  implies $|\bar{A}_r|>(1-2\tilde\eta) n$. As
  $|M_r|\ge(1-\tilde\eta)n$ it follows that $|B_r|>(1-3\tilde\eta)n$ and
  $|C_r|>(1-3\tilde\eta)n$. By  definition all edges between
  $\bar{A}_r$ and $B_r\cup C_r$
  are green and thus $K$ is in pyramid configuration with tunnel, pyramids
  $(B_r,\bar{A}_r)$ and $(C_r,\emptyset)$, and parameter $3\tilde\eta<\eta'$,
  which is a contradiction.
  
  Next we strengthen the last fact by showing that
  at most one of the sets $\bar{D}_c$ with $D\in\{A,B,C\}$ and
  $c\in\{r,g\}$ is significant.
  \begin{fact}\label{fac:odd:3}
  There is at most one set $D\in\{A,B,C\}$ and colour $c\in\{r,g\}$ such that
  $|\bar D_c|\ge\tilde\eta n$.
  \end{fact}
If such a $D$ and $c$ exist we assume, without loss of
generality, $D=A$ and $c=r$.
  Hence, for the proof of Fact~\ref{fac:odd:3}, assume that
  $|\bar{A}_{r}|\ge\tilde\eta n$. First we  show that
  \begin{equation}\label{eq:odd:bar}
    |\bar{B}_r|,|\bar{C}_r| < \tfrac{\tilde\eta}2 n.
  \end{equation}
  Assume for a contradiction and without loss of generality that
  $|\bar{B}_r|\ge \frac{\tilde{\eta}}2n$. By definition, all edges in
  $E(\bar{A}_r, C_r\dcup B_r)$ and $E(\bar{B}_r, C_r\dcup A_r)$ are green. Since
  $|\bar{A}_r|,|\bar{B}_r|\ge\tilde\eta n>2\eta n$ by assumption and
  $|A_r|,|B_r|\ge2\tilde\eta n/2>2\eta n$ (by Fact~\ref{fac:odd:2}) we can apply
  Proposition~\ref{prop:tool}\ref{prop:conn} to infer that the graph with
  edges $E(\bar{A}_r, C_r\dcup B_r)$ and $E(\bar{B}_r, C_r\dcup A_r)$ is
  connected. As $M_r$ is even we conclude that all edges in $E(A_r,C_r)$,
  $E(B_r,C_r)$, and $E(A_r,B_r)$ are red. Since $|A_r|,|B_r|,|C_r|\ge\tilde\eta
  n>2\eta n$ by Fact~\ref{fac:odd:2} we infer from  Proposition~\ref{prop:tool}\ref{prop:triangle} that  
  the graph $K[A_r,B_r,C_r]\subset R$ contains a red triangle
  which contradicts the fact that $M_r$
  is even.
  
  Thus it remains to show that $|\bar{D}_g|<\tilde\eta n$ for all
  $D\in\{A,B,C\}$. By~\eqref{eq:odd:bar} and Fact~\ref{fac:odd:2} we have
  $|B_r\cap B_g|, |C_r\cap C_g|>\frac{\tilde\eta}2 n>\eta n$ which implies
  that there is an edge in $E(B_r\cap B_g,C_r\cap C_g)$. By assumption we also have
  $|\bar{A}_r|\ge\tilde\eta n>2\eta n$ and thus each pair of vertices in
  $B_r\dcup C_r$ has a common neighbour in~$\bar{A}_r$ by
  \ref{prop:cherry} of Proposition~\ref{prop:tool}. By definition of $\bar{A}_r$
  all edges in $E(\bar{A}_r, B_r\dcup C_r)$ are green, and therefore we conclude that all
  edges in $E(B_r\cap B_g,C_r\cap C_g)$ are red since otherwise there would be
  a green triangle connected to $M_g$. Accordingly
  $|\bar{A}_{g}|\le 2\eta n<\tilde\eta n/2$ since otherwise we could equally
  argue that all edges in $E(B_r\cap B_g,C_r\cap C_g)$ are green, a contradiction.  Therefore
  $|A_g|\ge(1-\tilde\eta/2)n$. 
  As $|\bar{A}_r|\ge\tilde\eta n$ this implies
  $|A_g\cap\bar{A}_r|\ge\tilde\eta/2 n>\eta n$ and from~\eqref{eq:odd:bar} we
  also get $|B_r\cap\bar{B}_g|\ge\tilde\eta/2 n>\eta n$. Thus there is
  an edge in $E(A_g\cap\bar{A}_r,B_r\cap\bar{B}_g)$. However, this edge can
  neither be red since it connects $\bar{A}_r$ and $B_r$, nor green since it
  connects $\bar{B}_g$ and $A_g$, a contradiction. Therefore
  $|\bar{B}_g|<\tilde\eta n$ and by symmetry also $|\bar{C}_g|<\tilde\eta n$
  which finishes the proof of Fact~\ref{fac:odd:3}. 
    
  We label the vertices in each of the bipartite graphs $R$ and $G$ 
  according to their bipartition class by~$1$ and~$2$. In the remaining part
  of the proof we examine the distribution of these bipartition classes over the
  partition classes of~$K$.  
  Let $F_{ij}$ denote
  the set of vertices in $V(R)\cap V(G)$ with label $i$ in $R$ and label
  $j$ in $G$ for $i,j\in[2]$. Let further $F_{0j}$ be the set of vertices in
  $\bar{A}_r\cap V(G)$ that have label $j$ in $G$ for $j\in[2]$.
  Next we observe that each of the sets $F_{ij}$ with $i,j\in[2]$ is
  essentially contained in one partition class of~$K$.
  \begin{fact}\label{fac:odd:4}
    For all $i,j\in[2]$ there is at most
    one partition class $D\in\{A,B,C\}$ of~$K$
    with $|F_{ij}\cap D|\ge\tilde\eta n$. Moreover
    $E(F_{0j},F_{ij})=\emptyset$.
  \end{fact}
  To prove the first part of Fact~\ref{fac:odd:4} assume for a contradiction
  that $|F_{ij}\cap A|,|F_{ij}\cap B|\ge\tilde\eta n$. Then there would be an edge in
 $K[A\cap F_{ij},B\cap F_{ij}]$ since $\tilde\eta>\eta$. This contradicts the
 fact that $F_{ij}$ is independent by definition.
  For the second part observe that an edge in $E(F_{0j},F_{ij})$ can neither be
  red as such an edge would connect vertices from~$\bar{A}_r$ to~$R$ nor
  green since $F_{0j}\cup F_{ij}$ lies in one bipartition class~$j$ of~$G$.
  
  \begin{fact}\label{fac:odd:5}
  There are $X, Y\in \{A,B,C\}$ with $X\neq Y$ and indices
  $b,b',c,c'\in [2]$ with $bb'\neq cc'$ such that $|F_{bb'}\cap
    X|,|F_{cc'}\cap Y|\ge(1-5\tilde\eta)n$ and
    $|F_{0b'}|,|F_{0c'}|\le\tilde\eta n$.
  \end{fact}
We divide the proof of this fact into three cases: The first case deals with
$\bar A_r\neq \emptyset$, the second one with $\bar A_r=\emptyset$ and the
additional assumption that there are $D\in \{A,B,C\}$ and $ij\neq i'j'\in [2]$
such that $|D\cap F_{ij}|,|D\cap F_{i'j'}|\geq \tilde \eta n$. The third and
remaining case treats the situation when $\bar A_r=\emptyset$ and for each
$D\in \{A,B,C\}$ there is at most one index pair $(i,j)$ with $|D\cap
F_{ij}|\geq \tilde \eta n$.

For the first case, let $j\in [2]$ be such that $F_{0j}\neq \emptyset$. Observe
that then the second part of Fact~\ref{fac:odd:4} implies that $|F_{1j}\cap
(B\cup C)|,|F_{2j}\cap (B\cup C)|<\eta n$. Let $c'=b'\in [2]$ with $c'\neq j$.
Then, because Fact~\ref{fac:odd:3} implies that $|\bar B_r|, |\bar B_g|, |\bar
C_r|,|\bar C_g|<\tilde \eta n$, we have that $|B\cap (F_{1b'}\cup F_{2b'})|\ge
(1-4\tilde \eta)n$ and $|C\cap (F_{1c'}\cup F_{2c'})|\ge
(1-4\tilde \eta)n$. Thus there is a $b\in [2]$ such that $|B\cap
F_{bb'}|\ge \tilde \eta n$. Let $c'\in [2]$ with $c'\neq b'$. The first part of
Fact~\ref{fac:odd:4} implies that $|C\cap F_{bc'}|<\tilde \eta n$, thus
$|C\cap F_{cc'}|\ge (1-5\tilde \eta)n\ge \tilde \eta n$. By symmetry we also
get $|B\cap F_{bb'}|\ge (1-5\tilde\eta)n $. This proves the first part of
the statement for the first case. To see the second part, observe
that if $F_{0b'}\neq \emptyset$, then $|F_{1b'}\cap (B\cup C)|,|F_{2b'}\cap
(B\cup C)|<\eta n$ by Fact~\ref{fac:odd:4}, a contradiction.

The second part of the second and third cases is straightforward as
$F_{0,1},F_{0,2}\subseteq \bar A_r=\emptyset$. 
To see the first part of the second case let $D$ be as specified above and 
$\{X,Y\}= \{A,B,C\}\setminus
\{D\}$.  The first part of Fact~\ref{fac:odd:4} implies that $|F_{ij}\cap X|,
|F_{i'j'}\cap X|,|F_{ij}\cap Y|,|F_{i'j'}\cap Y|<\tilde \eta n$. Thus
$|(F_{ij'}\cup F_{i'j})\cap X|\ge (1-2\tilde \eta)n-2\tilde \eta n$, as $|\bar
X_r|,|\bar X_g|<\tilde \eta n$. Without loss of generality, let $ij'$ be such
that $|X\cap F_{ij'}|\ge \tilde \eta n$. We set $b:= i$, $b':=j'$, $c=i'$ and
$c':=j$. The rest of the proof is similar to the first case, proving that then
$|Y\cap F_{cc'}|\geq (1-5\tilde \eta)n$ and by symmetry that $|X\cap F_{bb'}|\ge
(1-5\tilde \eta )n$.

It remains to prove the first part of the third case. For this observe that
for all $D\in \{A,B,C\}$ we have that $|D\cap \bigcup_{(i'j')\neq
(i,j)}F_{i',j'}|<3\tilde \eta n$, where $i,j$ are as specified in the
definition of the third case. Observe also that $|\bar D_r|,|\bar D_g|<\tilde \eta n$. This
implies $|D\cap F_{i,j}|\ge (1-5\tilde \eta)n$, as desired. Hence, for
$X=B$ and $Y=C$ we obtain indices $b,b',c,c'$ such that $|X\cap F_{bb'}|,|Y\cap
F_{i'j'}|\ge (1-5\tilde \eta)n$, with $bb'\neq cc'$ by
Fact~\ref{fac:assign:cut}.

  This brings us to the last step which shows that $K$ is extremal, a
  contradiction.

  \begin{fact}
    $K$ is in pyramid configuration with parameter $\eta'$.
  \end{fact}

Let $X,Y\in \{A,B,C\}$ and $b,b',c,c'\in[2]$ be as in Fact~\ref{fac:odd:5}. Let
$Z\in \{A,B,C\}\setminus \{X,Y\}$. Assume without loss of generality that
$b=b'=1$. Thus Fact~\ref{fac:odd:5} states that $|F_{11}\cap
X|\ge(1-5\tilde\eta)n$ and $|F_{01}|\le\tilde\eta n$. We distinguish two cases. 
First, assume that $c'=2$ and set $\bar{c}:=3-c$. By
  Fact~\ref{fac:odd:5} this implies $|F_{c2}\cap Y|\ge(1-5\tilde\eta)n$ and
  $|F_{02}|\le\tilde\eta n$ and thus $|(F_{\bar{c}2}\cup F_{21})\cap
  Z|\ge(1-5\tilde\eta)n$ by Fact~\ref{fac:odd:4}. Moreover $E(F_{11}\cap
  X,F_{21}\cap Z)$ forms an $\eta$-complete red bipartite graph since
  $F_{11}\cup F_{21}$ is an independent set in $G$. Similarly $E(F_{c2}\cap
  Y,F_{\bar{c}2}\cap Z)$ forms an $\eta$-complete red bipartite graph. Further,
  if $c=2$ then $E(F_{c2}\cap Y,F_{21}\cap Z)$ and $E(F_{11}\cap
  X,F_{\bar{c}2}\cap Z)$ form $\eta$-complete green bipartite graphs (leading
  to crossings) and if $c=1$ then $E(F_{11}\cap X,F_{c2}\cap Y)$ forms an
  $\eta$-complete green bipartite graph (leading to a tunel). Therefore, in
  both subcases, $K$ is in pyramid configuration with parameter
  $5\tilde\eta\le\eta'$ and pyramids $(F_{11}\cap X,F_{21}\cap Z)$ and
  $(F_{c2}\cap Y,F_{\bar{c}2}\cap Z)$, unless one of the sets $F_{21}\cap Z$
  and $F_{\bar{c}2}\cap Z$ has size at most $10\tilde\eta n$. In this case,
  however, we can simply replace this set by the empty set and still obtain a
  pyramid configuration with parameter at most $15\tilde\eta\le\eta'$.
  
  In the case $c'=1$ we have $c=2$.  Fact~\ref{fac:odd:5} guarantees that
  $|F_{21}\cap Y|\ge(1-5\tilde\eta)n$. Since $|F_{01}|\le\tilde\eta n$ we
  conclude from Fact
  ~\ref{fac:odd:4} that 
  $|(F_{12}\cup F_{22} \cup F_{02})\cap Z|\ge(1-5\tilde\eta)n$. Similarly as
  before $E(F_{11}\cap X,(F_{12}\cup F_{02})\cap Z)$ and $E(F_{21}\cap
  Y,F_{22}\cap Z)$ form $\eta$-complete green bipartite graphs and $E(F_{11}\cap X,
  F_{21}\cap Y)$ forms an $\eta$-complete red bipartite graph. 
  Accordingly we also get a pyramid configuration with parameter
  $5\tilde\eta\le\eta'$ in this case, where the pyramids are $(F_{11}\cap X,
  (F_{12}\cup F_{02})\cap Z)$ and $(F_{21}\cap Y,F_{22}\cap Z)$ unless, again,
  $(F_{12}\cup F_{02})\cap Z$ or $F_{22}\cap Z$ are too small in which case we
  proceed as above.
\end{proof}


\subsection{Extremal configurations}
\label{sec:ext}

Our aim in this section is to provide a proof of Lemma~\ref{lem:ext}. This proof
naturally splits into two cases concerning pyramid and spider configurations,
respectively. The former is covered by Proposition~\ref{prop:pyra}, the
latter by Proposition~\ref{prop:spider}.

\begin{proposition}
\label{prop:pyra}
  Lemma~\ref{lem:ext} is true for pyramid configurations.
\end{proposition}
\begin{proof}
\setcounter{fact}{0}
  Given $\eta'$ set $\eta=\eta'/3$. Let $K$ be a
  coloured graph from $\K{\eta}$ that is in pyramid configuration with parameter
  $\eta$ and pyramids $(D_1,D'_1)$ and $(D_2,D'_2)$ such that the
  requirements of~\ref{def:pyra} in Definition~\ref{def:ext} are met for colours
  $c$ and $c'$.

  \begin{fact}\label{fac:pyra:cross}
    If the pyramid configuration has crossings then $K$ is
    $\big((1-\eta')n,(1-\eta')\frac32 n,2\big)$-good.
  \end{fact}

  Indeed, by Proposition~\ref{prop:match} there is a matching $M$ of
  colour either $c$ or $c'$ and size at least $(1-2\eta)\frac12n$ in
  $K[D_1,D_2]$. Note further, that the pyramid configuration with
  crossings is symmetric with respect to the colours $c$ and $c'$ and hence we
  may suppose, without loss of generality, that~$M$ is of colour $c$ and that
  $|D'_1|\ge(1-\eta)\frac12n$. As $K[D_1,D_1']$ and $K[D_2,D_2']$ are $(\eta,c)$-complete,
  there are $c$-coloured matchings $M_1$ and $M_2$ in $K[D_1,D_1']$ and
  $K[D_2\setminus M,D'_2]$, respectively, of size at least
  $\min\{|D'_1|,|D_1|\}-\eta n$ and $\min\{|D'_2|,|D_2\setminus M|\}-\eta n$,
  respectively. This implies 
  \[
     |M|+|M_1|+|M_2| \ge (1-3\eta)\frac32n
     = (1-\eta')\frac32n.
  \]
  Observe that, depending on the size of $M$, either $M\cup M_2$ or $M_1\cup
  M_2$ is a matching of size at least $(1-3\eta)n=(1-\eta')n$.  Now, the union
  of $M$, $M_1$, and $M_2$ forms a $2$-fork system $F$ and since $K[D_1,D_1']$
  and $K[D_2,D_2']$ are $(\eta,c)$-complete the bipartite graph formed by these
  two graphs and $M$ is connected and has partition classes $D_1\cup D'_2$
  and $D_2\cup D'_1$. It follows that $F$ has size $|M|+|M_1|+|M_2| \ge
  (1-\eta')\frac32n$.

  \begin{fact}\label{fac:pyra:tunnel1}
    If the pyramid configuration has a $c'$-tunnel and if 
    there is a matching $M$ of colour $c'$ and
    size at least $(1-\eta')\frac12n$ in $K[D_1,D'_1\cup D'_2]$ or
    in $K[D_2,D'_1\cup D'_2]$ then $K$ is $\big((1-\eta')n,(1-\eta')\frac32
    n,2\big)$-good in colour $c'$.
  \end{fact}
	As $K$ has a $c'$-tunnel, there is a connected matching $M'$ of colour $c'$ and
	size at least $|D_1|-\eta n \ge (1-\eta')n$ in $K[D_1,D_2]$. 
	We will extend the matching $M'$ (which is a $1$-fork-system) to a
	$2$-fork system.  Without loss of generality assume that the matching
	$M$ promised by Fact~\ref{fac:pyra:tunnel1} is in $K[D_1,D'_1\cup D'_2]$. As
	$M\cap D_1$ and $D_2$ are non-negligible the bipartite graph $K[M\cap
	D_1,D_2]$ is connected by Proposition~\ref{prop:tool}\ref{prop:conn} and
	thus $M$ is connected. Hence $M\cup M'$ forms a connected $2$-fork
	system centered in $D_1$ and of size $|M'|+|M|\geq (1-\eta')\frac 32n$.

  \begin{fact}\label{fac:pyra:tunnel2}
    If the pyramid configuration has a $c'$-tunnel but no crossings and
    there is no matching of colour $c'$ and size at least $(1-\eta')\frac 12n$
    in $K[D_1, D'_1\cup D'_2]$ or in $K[D_2,D'_1\cup D'_2]$
	then $K$ is $((1-\eta')n,(1-\eta')\frac32
    n,3)$-good in colour $c$.
  \end{fact}

  %
  %
  To obtain the $3$-fork
  system note that 
  Proposition~\ref{prop:match} implies
  that there are matchings $M_1$ and $M_2$ of colour $c$ and sizes at least
  $(1-\eta')\frac12n$ in $K[D_1,D'_1\cup D'_2]$ and $K[D_2,D'_1\cup D'_2]$,
  respectively.  The union of $M_1$ and $M_2$ forms a $2$-fork system $F$
  centered in $D'_1\cup D'_2$
  covering at least $(1-\eta')\frac12n$ vertices in $D_1$
  and at least $(1-\eta')\frac12n$ vertices in $D_2$.
  %
  We can assume without loss of
  generality that $|D'_1|\ge(1-\eta)\frac12n\ge(1-\eta')\frac12n$. 
  As $K[D_1,D_1']$ is $(\eta,c)$-complete and $|D_1\setminus
  F|\le(1-\eta')n-(1-\eta')\frac12n=(1-\eta')\frac12n$ we can greedily find a
  matching between $D'_1$ and $D_1\setminus F$ covering all but at most $\eta n$
  vertices of $D_1\setminus F$. Its union with $F$ forms a $3$-fork system $F'$
  centered in $D'_1\cup D'_2$ covering at least $(1-\eta')n$ vertices in $D_1$
  and at least $(1-\eta')\frac12n$ vertices in $D_2$, implying that $F'$ has
  size at least $(1-\eta')\frac32n$.
  The graph $K[D_1,D'_1]\cup K[D_2,D'_2]$ clearly contains a matching $M$ of
  size at least $|D_1'\cup D_2'|-\eta n\ge(1-\eta')n$ in colour~$c$.

   Since the pyramid configuration has no
  crossings there are edges of colour $c$ in $K[D_1,D'_2]\cup K[D_2,D'_1]$.
  Together with the fact that  $D_1$, $D_2$, $D'_1$, and  $D'_2$ are
  non-negligible, we obtain that the bipartite graphs $K[D_1,D'_1\cup D'_2]$ and
  $K[D_2,D'_1\cup D'_2]$ are connected by
  \ref{prop:conn} of Proposition~\ref{prop:tool}. 
  Thus the matching $M$ and the
  fork system $F'$ are both connected.
\end{proof}

\begin{proposition}
\label{prop:spider}
  Lemma~\ref{lem:ext} is true for spider configurations.
\end{proposition}
\begin{proof}
\setcounter{fact}{0}
  Given $\eta'$ set $\eta=\eta'/5$ and let $K$ be a
  coloured graph from $\K{\eta}$ that is in spider configuration with parameter
  $\eta$, i.\,e., it satisfies~\ref{def:spider} of Definition~\ref{def:ext}.
%
  In this proof we construct only matchings and fork systems of colour~$c$.
  Observe that these are connected by definition. We distinguish two cases.
  \smallskip
  
  {\sl Case~1:}\,
  First assume that $|A_1\cup B_1\cup
  C_1|<(1-\eta)\frac32n$. We will show that in this case our configuration
  contains both a connected matching of size at least $(1-\eta')n$ and a connected $3$-fork
  system of size at least $(1-\eta')\frac 32n$.
  We need the following auxiliary observation.
  
  \begin{fact}\label{fac:spider:ABempty}
    If $|A_1\cup B_1\cup C_1|<(1-\eta)\frac32n$ then $A_B=B_A=\emptyset$.
    Moreover $|A_1|+\eta n\ge|B_C|=|C_B|$ and $|B_1|+\eta n\ge|A_C|=|C_A|$.
  \end{fact}
  Indeed, by Condition~\ref{def:spider:3} of~\ref{def:spider} either
  $A_2=\emptyset$ and hence $A_B\subset A_2$ is empty or $|A_2\cup 
  B_2\cup(C_2\setminus C_C)|\le(1-\eta)\frac32n$. In the second case we
  conclude from $|A_1\cup B_1\cup C_1|<(1-\eta)\frac32n$ that
  \begin{equation*}
    |A_1\cup A_2|+|B_1\cup B_2|+|C_1\cup(C_2\setminus C_C)|<(1-\eta)3n.
  \end{equation*}
As $|A_1\cup A_2|,|B_1\cup B_2|,|C_1\cup C_2|\ge (1-\eta)n$ it follows
that $|C_1\cup (C_2\setminus C_C)|<(1-\eta)n$ and thus $C_C\neq \emptyset$.
  By Condition~\ref{def:spider:2} of~\ref{def:spider} we get
  $A_B=\emptyset$.  
  For the second part of the fact observe that Condition~\ref{def:spider:1}
  of~\ref{def:spider} states that $n-|A_2|\ge|B_C|=|C_B|$ and thus we conclude
 $|A_1|\ge(1-\eta)n-|A_2|\ge|B_C|-\eta n = |C_B|-\eta n$. The inequality
 $|B_1|\ge|A_C|-\eta n$ is established in the same way.

\begin{fact}\label{fac:spider:good1}
If $|A_1\cup B_1\cup C_1|<(1-\eta)\frac 32n$ then $K$ is
$\big((1-\eta')n,(1-\eta')\frac 32n),3\big)$-good.
\end{fact}
From Condition~\ref{def:spider:1} of~\ref{def:spider} we infer that
$|A_C|<n-|B_2|\le |B_1|+\eta n$ 
and Fact~\ref{fac:spider:ABempty} implies that
$|A_1|+|A_C|=|A_1|+|A_2|\ge(1-\eta)n$ and $|C_1|+|C_2|\ge(1-\eta)n$. We thus
conclude from $|A_1\cup B_1\cup C_1|<(1-\eta)\frac 32n$ that 
\begin{equation*}
|A_C|-\eta n<|B_1|<(1-\eta)\tfrac32n-|A_1 \cup C_1|
    <|A_C|+|C_2|-\eta n.
\end{equation*}
  This (together with the fact that $K[B_1,A_C]$ and $K[B_1,C_2]$ are
  $(\eta,c)$-complete) justifies that there is a $c$-coloured matching $M_1$ 
  in $K[B_1,A_C\cup C_2]$ covering all vertices of $B_1$ and all
  but at most~$\eta n$ vertices of $A_C$. Further, by
  Fact~\ref{fac:spider:ABempty} we know that $|B_C|\le|A_1|+\eta n$ and hence
  we can find a matching $M_2$ of colour $c$  in (the $(\eta,c)$-complete graph)
  $K[B_C,A_1]$ covering all but at most $\eta n$ vertices of $B_C$. 
  The matching $M:=M_1\cup M_2$ satisfies
  \begin{equation*}
   |M|\ge  |B_1|+|B_C|-\eta n=|B_1|+|B_2|-\eta n\ge(1-\eta)n-\eta
   n\ge(1-\eta')n,
  \end{equation*}
  where the equality follows from Fact~\ref{fac:spider:ABempty}.
  Next, we extend the matching $M$ to a connected $3$-fork system of colour $c$
  and size at least $(1-\eta')\frac 32n$ in the following way. Consider maximal
matchings $M_3$, $M_4$, and $M_5$ in $K[B_1,C_A\setminus M_1]$,
$K[A_1,C_B\setminus M_1]$  and $K[A_1,C_C\setminus M_1]$, respectively. By
Fact~\ref{fac:spider:ABempty} we infer that $M_3$ and $M_4$ each cover all but at most $\eta n$ vertices of $C_A\setminus M_1$ and
$C_B\setminus M_1$, respectively. As $|C_C|\le |C_C\cup C_1|\le |A_1|$
by Condition~\ref{def:spider:1} of~\ref{def:spider} the matching $M_5$ covers
all but at most $ \eta n$ vertices of $C_C$.

Then the  union $M\cup M_3\cup M_4\cup M_5$
is a 3-fork-system $F$ centered in $A_1\cup B_1$ and covering all but at most $5\eta
n$ vertices of $A_C\cup B_C\cup C_A\cup C_B\cup C_C=A_2\cup B_2\cup C_2$. Thus
$F$ has size at least $(1-\eta)3n-|A_1\cup B_1\cup C_1|-5\eta n\geq (1-\eta')\frac
32 n$. 
\smallskip

{\sl Case~2:}\, Now we turn to the case  $|A_1\cup B_1\cup C_1|\ge
(1-\eta)\frac 32n$. We further divide this case into two subcases, treating
$C_1=\emptyset$ and $C_1\neq\emptyset$, respectively.

  \begin{fact}\label{fac:spider:C2fork}
    If $|A_1\cup B_1\cup C_1|\ge(1-\eta)\frac32n$ and $C_1=\emptyset$ then
    $K$ is $\big((1-\eta')n,(1-\eta')\frac 32n,2\big)$-good.
  \end{fact}
  By definition $|C_2|\ge(1-\eta)n-|C_1|=(1-\eta)n$ in
  this case. Therefore, using the fact that $K[A_1,C_2]$ and $K[B_1,C_2]$
  are $(\eta,c)$-complete, we can greedily construct a maximal matching $M_A$ in
  $K[A_1,C_2]$ and a maximal matching $M'_B$ in $K[B_1,C_2\setminus M_A]$
  such that the matching $M:=M_A\cup M'_B$ covers $C_2$ (as $|A_1\cup
  B_1|=|A_1\cup B_1\cup C_1|> |C_2|+\eta n$) and thus has size at
  least $(1-\eta)n$. Then we extend $M'_B$ to a maximal
  matching $M_B$ in $K[B_1,C_2]$. Observe that $M_A$  and $M_B$ cover all but
  at most $\eta n$ vertices of $A_1$ and $B_1$, respectively. Thus the
  $2$-fork system $F:=M_A\cup M_B$ has size at least $|A_1\cup B_1|-2\eta
  n=|A_1\cup B_1\cup C_1|-2\eta n\ge (1-\eta')\frac32n$. 
  
  Now consider the subcase when $C_1\neq\emptyset$.

\begin{fact}\label{fac:spider:BC}
If $|A_1\cup B_1\cup C_1|\ge (1-\eta)\frac 32n$ and $C_1\neq \emptyset$ then
$|B_1\cup C_1|\le (1-\eta)\frac 34n$ and we have $|C_2|\ge (1-\eta)\frac
14n$ and $|C_1|\le |B_2|-\eta n$.
\end{fact}
The first inequality follows from Condition~\ref{def:spider:4}
of~\ref{def:spider}. Accordingly $|C_2|\ge (1-\eta)n-|C_1|\ge (1-\eta)\frac
14n$ which establishes the second inequality. For the third inequality we use
that $|B_1\cup B_2|\ge (1-\eta)n$ by definition and so
\begin{equation*}
    |C_1|\le(1-\eta)\tfrac34 n-|B_1|
    \le(1-\eta)\tfrac34 n - (1-\eta)n + |B_2|
    \le|B_2|-\eta n.
\end{equation*}

\begin{fact}\label{fac:spider:match}
If $|A_1\cup B_1\cup C_1|\ge (1-\eta)\frac 32n$ and $C_1\neq \emptyset$ then
there is a matching $M$ of size at least
$(1-\eta)n$ and colour $c$ covering $C_1$.
\end{fact}
  
Let $M_1$ be a maximal matching in $K[C_1,B_2]$. We conclude from
Fact~\ref{fac:spider:BC} that $M_1$ covers $C_1$. Let~$M_2$ be a maximal
matching in $K[C_2,A_1\cup B_1]$. As $|C_2|\le n-|C_1|\leq |A_1\cup B_1|-\eta
n$ the matching~$M_2$ covers $C_2$. Setting $M:=M_1\cup M_2$, we obtain a
matching of size $|M|=|C_1|+|C_2|\ge (1-\eta)n$ as required.
  
\begin{fact}\label{fac:spider:fork}
If $|A_1\cup B_1\cup C_1|\ge (1-\eta)\frac 32n$ and $C_1\neq \emptyset$, then
there is a $3$-fork system of colour $c$ and of size at least $(1-\eta')\frac
32n$.
\end{fact}  
Let $M$ be the matching from Fact~\ref{fac:spider:match}. 
Clearly, we can greedily construct a $2$-fork
system $F'$ in the $(\eta,c)$-complete graph $K[C_2,(A_1\cup B_1)\setminus M]$
which either is of size $2|C_2|$ or covers all but at most $\eta n$ vertices
of $|(A_1\cup B_1)\setminus M|$.
Then $F:=M\cup F'$ forms a $3$-fork
system. If the former case occurs we infer from
Fact~\ref{fac:spider:BC} that $F$ is of size at least
$(1-\eta)n+2|C_2|\ge (1-\eta)\frac 32n$. In the latter case $F$
covers all but at most $\eta n$ vertices of $A_1\cup B_1\cup C_1$ and thus has
size at least $(1-\eta)\frac 32n-\eta n\ge (1-2\eta)\frac 32n$.
We conclude that $K$ is $((1-\eta')n,(1-\eta')\frac 32n,3)$-good also in the
subcase $|A_1\cup B_1\cup C_1|\ge(1-\eta)\frac 32n$ and $C_1\neq \emptyset$.
\end{proof}


\section{Concluding remarks}
\label{sec:remarks}

  As noted earlier our proof of Theorem~\ref{thm:main} applies to
  suitably chosen (sparser) subgraphs of~$K_{n,n,n}$ as well. More precisely,
  for any fixed $p\in(0,1)$ the same method can be used to show that asymptotically
  almost surely $\mathcal G_p(n,n,n)\rightarrow \mathcal{T}_t^{\Delta}$, where
  $\mathcal G_p(n,n,n)$ is a random tripartite graph with edge probability $p$
  and partition classes of size~$n$, and where $t\le(1-\mu)n/2$ and $\Delta\le
  n^\alpha$ for a small positive $\alpha=\alpha(\mu,p)$.
  Indeed,  standard methods can be used to show that the following holds
  asymptotically almost surely for $G=\mathcal G_p(n,n,n)$ with  partition classes $V_1\dcup
  V_2\dcup V_3$ and for any $\zeta>0$:
  \begin{itemize}[leftmargin=*]
    \item $G$ has at most $4pn^2$ edges.
    \item $e(U,W)\ge p|U||W|/2$ for all $U\subset V_i$ and $W\subset V_j$,
      $i\neq j$, with $\min\{|U|,|W|\}>\zeta n$.
  \end{itemize}
  The first property guarantees that we obtain a graph with few edges.
  We claim further that these two properties imply that
  $G\rightarrow \mathcal{T}^{\Delta}_k$. To see this we proceed as in the proof
  of Theorem~\ref{thm:main} and apply the regularity lemma
  on the coloured graph~$G$. We then colour an edge in the reduced graph
  $\mathbb G$ by green or red, repectively, if the corresponding cluster pair
  is regular and has density at least $p/4$ in green or red. Using the two
  properties from above it is not difficult to verify that $\mathbb G$ is a
  coloured tripartite graph that is $\eta$-complete. Hence, from this point
  on, we can use the strategy described in the proof of Theorem~\ref{thm:main},
  apply our structural lemma, Lemma~\ref{lem:good-odd}, the assignment lemma,
  Lemma~\ref{lem:valid}, and the embedding lemma, Lemma~\ref{lem:emb}.

   One may ask whether this approach can be pushed even further and consider
   random tripartite graphs $\mathcal G_p(n,n,n)$ with
   edge probabilities $p(n)$ that tend to zero as $n$ goes to infinity. 
   It is likely that similar methods can be used in this case in conjunction
   with the regularity method for sparse graphs
   (see, e.g.,\ \cite{kohayakawa00:_regul_i}). 
   \smallskip

   We close with an extension of Schelp's conjecture that was
   suggested to us by Ji\v r\'i Matou\v sek.
  \begin{question}\label{conjM}
    Is it true that for all $\Delta\in\mathbb{N}$ and $\mu>0$ there is a
    $n_0\in\mathbb{N}$ such that the following holds for all
    $n\ge n_0$? If $t\le(1-\mu)\frac12 n$ and
    $G$ is a graph on~$n$ vertices with minimum degree
    $\delta(G)\ge(\frac23-\mu)n$ then
    $G\rightarrow\mathcal{T}^\Delta_t$.
  \end{question}


\bibliographystyle{amsplain} \bibliography{bibl}


\end{document}